%% file: conformalrestriction_bm.tex
\documentclass{imsart}

\usepackage{color}
\usepackage{helvet}         
\usepackage{courier}        
\usepackage{type1cm}        
\usepackage{makeidx}         
\usepackage{graphicx}        
\usepackage{multicol}        
\usepackage[bottom]{footmisc}
\usepackage{amsmath}
\usepackage{amsfonts}
\usepackage{amssymb}
\usepackage{amsthm}
\usepackage{comment}
\usepackage{subcaption}
\newtheorem{theorem}{Theorem}
\newtheorem{corollary}[theorem]{Corollary}
\newtheorem{lemma}[theorem]{Lemma}

\newtheorem{proposition}[theorem]{Proposition}

\newtheorem{remark}[theorem]{Remark}
\newtheorem{definition}[theorem]{Definition}

\numberwithin{theorem}{section}
\numberwithin{figure}{section}
\numberwithin{equation}{section}

\begin{document}
\newcommand{\eps}{\epsilon}
\newcommand{\ov}{\overline}
\newcommand{\U}{\mathbb{U}}
\newcommand{\T}{\mathbb{T}}
\newcommand{\HH}{\mathbb{H}}
\newcommand{\LA}{\mathcal{A}}
\newcommand{\LF}{\mathcal{F}}
\newcommand{\LK}{\mathcal{K}}
\newcommand{\LE}{\mathcal{E}}
\newcommand{\LL}{\mathcal{L}}
\newcommand{\LU}{\mathcal{U}}
\newcommand{\LV}{\mathcal{V}}
\newcommand{\R}{\mathbb{R}}
\newcommand{\C}{\mathbb{C}}
\newcommand{\N}{\mathbb{N}}
\newcommand{\Z}{\mathbb{Z}}
\newcommand{\E}{\mathbb{E}}
\newcommand{\PP}{\mathbb{P}}
\newcommand{\QQ}{\mathbb{Q}}
\newcommand{\A}{\mathbb{A}}
\newcommand{\bn}{\mathbf{n}}
\newcommand{\cond}{\,|\,}

\begin{frontmatter}

\title{Conformal Restriction and Brownian Motion}
\runtitle{Conformal Restriction and Brownian Motion}

\begin{aug}
  \author{Hao Wu\ead[label=e1]{hao.wu.proba@gmail.com}}
 
  \address{NCCR/SwissMAP, Section de Math\'ematiques, Universit\'e de Gen\`eve, Switzerland\\ 
           \printead{e1}}

  \affiliation{NCCR/SwissMAP, Section de Math\'ematiques, Universit\'e de G\`eneve, Switzerland}

\end{aug}

\begin{abstract}
This survey paper is based on the lecture notes for the mini course in the summer school at Yau Mathematics Science Center, Tsinghua University, 2014.

We describe and characterize all random subsets $K$ of simply connected domain which satisfy the ``conformal restriction" property. There are two different types of random sets: the chordal case and the radial case. In the chordal case, the random set $K$ in the upper half-plane $\HH$ connects two fixed boundary points, say 0 and $\infty$, and given that $K$ stays in a simply connected open subset $H$ of $\HH$, the conditional law of $\Phi(K)$ is identical to that of $K$, where $\Phi$ is any conformal map from $H$ onto $\HH$ fixing 0 and $\infty$. In the radial case, the random set $K$ in the upper half-plane $\HH$ connects one fixed boundary points, say 0, and one fixed interior point, say $i$, and given that $K$ stays in a simply connected open subset $H$ of $\HH$, the conditional law of $\Phi(K)$ is identical to that of $K$, where $\Phi$ is the conformal map from $H$ onto $\HH$ fixing 0 and $i$. 

It turns out that the random set with conformal restriction property are closely related to the intersection exponents of Brownian motion. The construction of these random sets relies on Schramm Loewner Evolution with parameter $\kappa=8/3$ and Poisson point processes of Brownian excursions and Brownian loops.
\end{abstract}



\tableofcontents

\end{frontmatter}

\newpage
\input{tex/foreword.tex}

\newpage
\section{Brownian intersection exponents and conformal restriction property}\label{sec::introduction}
\input{tex/introduction.tex}

\newpage
\section{Brownian motion, excursion and loop}\label{sec::bm}
\input{tex/brownian_motion_excursion_loop.tex}

\newpage
\section{Chordal SLE}\label{sec::chordal_sle}
\input{tex/chordal_sle.tex}
\newpage
\section{Chordal conformal restriction}\label{sec::chordal_restriction}
\input{tex/chordal_restriction.tex}
\newpage
\section{Radial SLE}\label{sec::radial_sle}
\input{tex/radial_sle.tex}

\newpage
\section{Radial conformal restriction}\label{sec::radial_restriction}
\input{tex/radial_restriction.tex}
\newpage
\bibliographystyle{alpha}
\bibliography{hao_wu_thesis}

\end{document}

%% file: tex/foreword.tex
\noindent\textbf{Foreword}

The goal of these lectures is to review some of the results related to conformal restriction: the chordal case and the radial case. The audience of the summer school of Yau Mathematical Science Center, Tsinghua University, consists of senior undergraduates and graduates. Therefore, I assume knowledge in stochastic calculus (Brownian motion, It\^o formula etc.) and basic knowledge in complex analysis (Riemann's Mapping Theorem etc.).

These lecture notes are not a compilation of research papers, thus some details in the proofs are omitted. Also partly because of the limited number of lectures, I chose to focus on the main ideas of the proofs. Whereas, I cite the related papers for interested readers.

Of course, I would like to thank my advisor Wendelin Werner with whom I learned the topic on conformal restriction, SLE and solved conformal restriction problem for the radial case. I want to express my gratitude to all participants of the course, as well as to an anonymous reviewer who have sent me their comments and remarks on the previous draft of these notes.

It has been a great pleasure and a rewarding experience to go back to Tsinghua University and to give a lecture here where I spent four years of undergraduate. I owe my thanks to Prof. Yau and Prof. Poon for giving me the chance.
\medbreak
\noindent\textbf{Outline}

In Section \ref{sec::introduction}, I will briefly describe Brownian intersection exponents and conformal restriction property. The results are collected from \cite{LawlerWernerIntersectionExponents, LawlerWernerUniversalityExponent, LawlerSchrammWernerExponent1, LawlerSchrammWernerExponent2, LawlerSchrammWernerExponent3}. In fact, Brownian intersection exponents have close relation with Quantum Field Theory and the interested readers could consult \cite{LawlerWernerIntersectionExponents, DuplantierKwonConformalInvariance} and references there for more background and motivation. 
Section \ref{sec::bm} is a review on Brownian path: Brownian motion, Brownian excursion and Brownian loop. The results are collected from \cite{LawlerWernerBrownianLoopsoup, WernerConformalRestrictionRelated, WernerSelfavoidingLoop, WernerSelfavoidingLoop, SheffieldWernerCLE, LawlerWernerBrownianLoopsoup}. Section \ref{sec::chordal_sle} is an introduction on chordal SLE. Since I only need SLE$_{8/3}$ in the following of the lecture, I focus on simple SLE paths, i.e. $\kappa\in [0,4]$. For a more complete introduction on SLE, I recommend the readers to read the lecture note by Wendelin Werner \cite{WernerRandomPlanarcurves} or the book by Gregory Lawler \cite{LawlerConformallyInvariantProcesses}. Section \ref{sec::chordal_restriction} is about the chordal conformal restriction property. The results are collected from \cite{LawlerSchrammWernerConformalRestriction}. Section \ref{sec::radial_sle} is an introduction on radial SLE and again, for a more complete introduction on radial SLE, please read \cite{WernerRandomPlanarcurves, LawlerConformallyInvariantProcesses}. Section \ref{sec::radial_restriction} is about the radial restriction property. The results are contained in \cite{WuConformalRestrictionRadial}.
\medbreak
\noindent\textbf{Notations}
\[\HH=\{z\in\C: \Im(z)>0\}.\]
\[\U=\{z\in\C: |z|<1\},\quad \U(w,r)=\{z\in\C: |z-w|<r\}.\]
Denote \[f(\eps)\approx g(\eps)\quad \text{as } \eps\to 0, \quad \text{if}\quad \lim_{\eps\to 0}\frac{\log f(\eps)}{\log g(\eps)}=1.\]

%% file: tex/introduction.tex
\subsection{Intersection exponents of Brownian motion}
Probabilists and physicists are interested in the property of intersection exponents for two-dimensional Brownian motion (BM for short). Suppose that we have $n+p$ independent planar BMs: $B^1,\cdots, B^n$ and $W^1,\cdots, W^p$. $B^1,\cdots, B^n$ start from the common point $(1,1)$ and $W^1,\cdots, W^p$ start from the common point $(2,1)$. We want to derive the probability that the paths of \[B^j,j=1,\cdots,n \text{ up to time } t\] and the paths of \[W^l,l=1,\cdots, p\text{ up to time } t\] do not intersect. Precisely,
\[f_{n,p}(t):=\PP\left[\bigcup_{j=1}^n B^j[0,t]\ \bigcap\ \bigcup_{l=1}^p W^l[0,t]=\emptyset\right].\]

We can see that this probability decays as $t\to\infty$ roughly like a power of $t$,\footnote{Why? Hint: $f_{n,p}(ts)\thickapprox f_{n,p}(t)f_{n,p}(s)$.}
and the $(n,p)$-\textbf{whole plane intersection exponent} $\xi(n,p)$ is defined by\footnote{Why $\sqrt{t}$: for BM $B$, the diameter of $B[0,t]$ scales like $\sqrt{t}$.}
\[f_{n,p}(t)\thickapprox \left(\frac{1}{\sqrt{t}}\right)^{\xi(n,p)}, \quad t\to\infty.\]
We say that $\xi(n,p)$ is the whole plane intersection exponent between one packet of $n$ BMs and one packet of $p$ BMs.

Similarly, we can define more general intersection exponents between $k\ge 2$ packets of BMs containing $p_1,...,p_k$ paths respectively:
\[B^j_l, \quad  l=1,..., p_j, \quad j=1,...,k.\] Each path in $j$th packet starts from $(j,1)$ and has to avoid all paths of all other packets. The $(p_1,...,p_k)$-whole-plane intersection exponent $\xi(p_1,...,p_k)$ is defined through, as $t\to\infty$,
\[\PP\left[\bigcup_{u=1}^{p_{j_1}} B_u^{j_1}[0,t]\ \bigcap\ \bigcup_{v=1}^{p_{j_2}} B_v^{j_2}[0,t]=\emptyset, 1\le j_1<j_2\le k\right]\thickapprox \left(\frac{1}{\sqrt{t}}\right)^{\xi(p_1,...,p_k)}.\]

Another important quantity is \textbf{half-plane intersection exponents} of BMs. They are defined exactly as the whole-plane intersection exponents above except that one adds one more restriction that all BMs (up to time $t$) remain in the upper half-plane $\HH:=\{(x,y)\in\R^2: y>0\}$. We denote these exponents by $\tilde{\xi}(p_1,...,p_k)$. For instance, $\tilde{\xi}(1,1)$ is defined by, as $t\to\infty$,
\[\PP\left[B[0,t]\bigcap W[0,t]=\emptyset, B[0,t]\subset \HH, W[0,t]\subset \HH\right]\thickapprox \left(\frac{1}{\sqrt{t}}\right)^{\tilde{\xi}(1,1)}.\]

These exponents also correspond to the intersection exponents of planar simple random walk \cite{MandelbrotFractalNature, BurdzyLawlerNonintersectionExponentRW1, LawlerDisconnectionExponentSRW, LawlerIntersectionExponentSRW}. 
Several of these exponents correspond to Hausdorff dimensions of exceptional subsets of the planar Brownian motion or simple random walk \cite{LawlerCutBrownian, LawlerDimensionFrontierBM}. 
Physicists have made some striking conjectures about these exponents \cite{DuplantierKwonConformalInvariance, DuplantierLawlerGeometryBM} and they are proved by mathematicians later \cite{LawlerWernerIntersectionExponents,
LawlerSchrammWernerAnalyticityIntersectionExponentsBM, LawlerWernerUniversalityExponent,  LawlerSchrammWernerExponent1, LawlerSchrammWernerExponent2, LawlerSchrammWernerExponent3}. We list some of the results here.
\begin{enumerate}
\item There is a precise natural meaning of the exponents $\xi(\mu_1,...,\mu_k)$ and $\tilde{\xi}(\lambda_1,...,\lambda_k)$ for positive real numbers $\mu_1,...,\mu_k, \lambda_1,...,\lambda_k$. (See Sections~\ref{sec::chordal_restriction}  and \ref{sec::radial_restriction}).
\item These exponents satisfy certain functional relations
\begin{enumerate}
\item Cascade relations:
\[\tilde{\xi}(\lambda_1,...,\lambda_{j-1},\tilde{\xi}(\lambda_j,...,\lambda_k))=\tilde{\xi}(\lambda_1,...,\lambda_k)\]
\[\xi(\mu_1,...,\mu_{j-1},\tilde{\xi}(\mu_j,...,\mu_k))=\xi(\mu_1,...,\mu_k)\]
\item Commutation relations:
\[\tilde{\xi}(\lambda_1,\lambda_2)=\tilde{\xi}(\lambda_2,\lambda_1),\quad \xi(\mu_1,\mu_2)=\xi(\mu_2,\mu_1)\]
\end{enumerate}
\item One can define a positive, strictly increasing continuous function $U$ on $[0,\infty)$ by
\[U^2(\lambda)=\lim_{N\to\infty}\tilde{\xi}(\underset{N}{\underbrace{\lambda,...,\lambda}})/N^2.\]
Then we have
\[U(\tilde{\xi}(\lambda_1,...,\lambda_k))=U(\lambda_1)+\cdots+U(\lambda_k).\]
This shows in particular that $\tilde{\xi}$ is encoded in $U$.
\item The whole-plane intersection exponent $\xi$ can be represented as a function of the half-plane intersection exponent $\tilde{\xi}:$
\[\xi(\mu_1,...,\mu_k)=\eta(\tilde{\xi}(\mu_1,...,\mu_k)).\] The function $\eta$ is called a generalized disconnection exponent and it is a continuous increasing function.
\item Physicists predict that 
\[\tilde{\xi}(\underset{N}{\underbrace{1,...,1}})=\frac{1}{3}N(2N+1),\quad \xi(\underset{N}{\underbrace{1,...,1}})=\frac{1}{12}(4N^2-1).\]
\end{enumerate}

Combining all these results, we could predict that:
\[U(1)=\sqrt{\frac{2}{3}},\quad U(\lambda)=\frac{\sqrt{24\lambda+1}-1}{\sqrt{24}}\]
\begin{equation}\label{eqn::wholeplane_exponent}\tilde{\xi}(\lambda_1,...,\lambda_k)=\frac{1}{24}\left((\sqrt{24\lambda_1+1}+\cdots+\sqrt{24\lambda_k+1}-(k-1))^2-1\right)\end{equation}
\[\eta(x)=\frac{1}{48}\left((24x+1)^2-4\right)\]
\begin{equation}\label{eqn::halfplane_exponent}\xi(\mu_1,...,\mu_k)=\frac{1}{48}\left((\sqrt{24\mu_1+1}+\cdots+\sqrt{24\mu_k+1}-k)^2-4\right)\end{equation}

\subsection{From Brownian motion to Brownian excursion}\label{subsec::intro_bm_be}
Consider the simplest exponent $\tilde{\xi}(1)=1$. Suppose $B$ is a planar BM started from $i$, then we have
\[\PP[B[0,t]\subset \HH]\thickapprox (\frac{1}{\sqrt{t}})^{\tilde{\xi}(1)}.\]
Suppose $W$ is a planar BM started from $\eps i$, then
\[\PP[W[0,t]\subset \HH]\thickapprox (\frac{\eps}{\sqrt{t}})^{\tilde{\xi}(1)},\]
since $(W(\eps^2t)/\eps, t\ge 0)$ has the same law as $B$. Consider the law of $W$ conditioned on the event ${W[0,t]\subset\HH}$, we can see that the limit as $t\to\infty,\eps\to 0$ exists. We call the limit as Brownian excursion and denote its law as $\mu^{\sharp}_{\HH}(0,\infty)$. There is another equivalent way to define $\mu^{\sharp}_{\HH}(0,\infty)$: Suppose $W$ is a planar BM started from $\eps i$, consider the law of $W$ conditioned on the event $[W\text{ hits }\R+iR\text{ before }\R]$. Let $R\to\infty,\eps\to 0$, the limit is the same as $\mu^{\sharp}_{\HH}(0,\infty)$. (We will discuss Brownian excursion in more detail in Section~\ref{sec::bm}).

Suppose $Z$ is a Brownian excursion, $A$ is a bounded closed subset of $\overline{\HH}$ such that $\HH\setminus A$ is simply connected and $0\not\in A$. 
Riemann's Mapping Theorem says that, if we fix three boundary points, there exists a unique conformal map from $\HH\setminus A$ onto $\HH$ fixing the three points. In our case, since $A$ is bounded closed and $0\not\in A$, any conformal map from $\HH\setminus A$ onto $\HH$ can be extended continuously around the origin and $\infty$. 
Let $\Phi_A$ be the unique conformal map from $\HH\setminus A$ onto $\HH$ such that
\[\Phi_A(0)=0,\quad \Phi_A(\infty)=\infty,\quad \Phi_A(z)/z\to 1\text{ as }z\to\infty.\]
Consider the law of $\Phi_A(Z)$ conditioned on $[Z\cap A=\emptyset]$. We have that, for any bounded function $F$,
\begin{align*}
\E&[F(\Phi_A(Z))\cond Z\cap A=\emptyset]\notag\\
&=\lim_{R\to\infty,\eps\to 0} \E[F(\Phi_A(W))\cond W\cap A=\emptyset, W\text{ hits }\R+iR\text{ before }\R]\tag{\text{W: BM started from $\eps i$}}\\
&=\lim_{R\to\infty,\eps\to 0} \frac{\E\left[F(\Phi_A(W)) 1_{\{W\cap A=\emptyset, W\text{ hits }\R+iR\text{ before }\R\}}\right]}{\PP[W\cap A=\emptyset, W\text{ hits }\R+iR\text{ before }\R]}\\
&=\lim_{R\to\infty,\eps\to 0} \frac{\E\left[F(\Phi_A(W)) 1_{\{W\text{ hits }\R+iR\text{ before }\R\}}\cond W\cap A=\emptyset\right]}{\PP[W\text{ hits }\R+iR\text{ before }\R\cond W\cap A=\emptyset]}.
\end{align*}
Conditioned on $[W\cap A=\emptyset]$, the process $\tilde{W}=\Phi_A(W)$ has the same law as a BM started from $\Phi_A(\eps i)$, thus
\begin{eqnarray*}
\lefteqn{\E[F(\Phi_A(Z))\cond Z\cap A=\emptyset]}\\
&=&\lim_{R\to\infty,\eps\to 0} \frac{\E[F(\tilde{W}) 1_{\{\tilde{W}\text{ hits }\Phi_A(\R+iR)\text{ before }\R\}}]}{\PP[\tilde{W}\text{ hits }\Phi_A(\R+iR)\text{ before }\R]}\\
&=&\E[F(Z)].
\end{eqnarray*}
In other words, the Brownian excursion $Z$ satisfies the following conformal restriction property: the law of $\Phi_A(Z)$ conditioned on $[Z\cap A=\emptyset]$ is the same as $Z$ itself.
Conformal restriction property is closely related to the half-plane/whole-plane intersection exponents.
\subsection{Chordal conformal restriction property}
\begin{definition}\label{def::collection_chordal}
Let $\LA_c$ be the collection of all bounded closed subset $A\subset\overline{\HH}$ such that
\[0\not\in A,\quad A=\overline{A\cap\HH}, \quad \text{ and }\HH\setminus A \text{ is simply connected}.\]
Denote by $\Phi_A$ the conformal map from $\HH\setminus A$ onto $\HH$ such that
\[\Phi_A(0)=0,\quad \Phi_A(\infty)=\infty,\quad \Phi_A(z)/z\to 1\text{ as }z\to\infty.\]
\end{definition}
We are interested in closed random subset $K$ of $\overline{\HH}$ such that
\begin{enumerate}
\item [(1)] $K\cap \R=\{0\}$, $K$ is unbounded, $K$ is connected and $\HH\setminus K$ has two connected components
\item [(2)] $\forall \lambda>0$, $\lambda K$ has the same law as $K$
\item [(3)] For any $A\in\LA_c$, we have that the law of $\Phi_A(K)$ conditioned on $[K\cap A=\emptyset]$ is the same as $K$.
\end{enumerate}
The combination of the above properties is called \textbf{chordal conformal restriction property}, and the law of such a random set is called \textbf{chordal restriction measure}. From Section \ref{subsec::intro_bm_be}, we see that the Brownian excursion satisfies the chordal conformal restriction property only except Condition (1). Consider a Brownian excursion path $Z$, it divides $\HH$ into many connected components and there is one, denoted by $C_+$, that has $\R_+$ on the boundary and there is one, denoted by $C_-$, that has $\R_-$ on the boundary. We define the ``fill-in'' of $e$ to be the closure of the subset $\HH\setminus (C_-\cup C_+)$. Then the ``fill-in" of $Z$ satisfies Condition (1), and the ``fill-in" of the Brownian excursion satisfies chordal conformal restriction property. Moreover, for $n\ge 1$, the the union of the ``fill-in" of $n$ independent Brownian excursions also satisfies the chordal conformal restriction property. Then one has a natural question: except Brownian excursions, do there exist other chordal restriction measure, and what are all of them? 

It turns out that there exists only a one-parameter family $\PP(\beta)$ of such probability measures for $\beta\ge 5/8$ \cite{ LawlerSchrammWernerConformalRestriction} and there are several papers related to this problem \cite{LawlerWernerUniversalityConformalInvariantIntersectionExponents, WernerConformalRestrictionRelated}. More detail in Sections~\ref{sec::chordal_sle} and \ref{sec::chordal_restriction}. 
The complete answer to this question relies on the introduction of SLE process \cite{SchrammFirstSLE}. In particular, an important ingredient is that SLE$_{8/3}$ satisfies chordal conformal restriction property. It is worthwhile to spend a few words on the specialty for SLE$_{8/3}$. In \cite{LawlerWernerUniversalityExponent}, the authors predicted a strong relation between Brownian motion, self-avoiding walks, and critical percolation. The boundary of the critical percolation interface satisfies conformal restriction property and the computations of its exponents yielded the Brownian intersection exponents. It is proved that the scaling limits of critical percolation interface is SLE$_6$ for triangle lattice \cite{SmirnovPercolationConformalInvariance} and the boundary of SLE$_6$ is locally SLE$_{8/3}$. 
Self-avoiding walk also exhibits conformal restriction property. It is conjectured \cite{LawlerSchrammWernerScalinglimitSAW} that the scaling limit of self-avoiding walk is SLE$_{8/3}$. All these observations indicate that SLE$_{8/3}$ is a key object in describing conformal restriction property. 

As expected, for $n\ge 1$,  the union of the ``fill-in" of $n$ independent Brownian excursions corresponds to $\PP(n)$ and, for $\beta\ge 5/8$, the measure $\PP(\beta)$ can viewed as the law of a packet of $\beta$ independent Brownian excursions.
The chordal restriction measures are closely related to half-plane intersection exponent (will be proved in Section~\ref{subsec::restriction_exponents_chordal}):
Suppose $K_1,...,K_p$ are $p$ independent chordal restriction samples of parameters $\beta_1,...,\beta_p$ respectively. The ``fill-in" of the union of these sets
\[\bigcup_{j=1}^p K_j\] conditioned on the event (viewed as a limit)
\[[K_{j_1}\cap K_{j_2}=\emptyset, 1\le j_1<j_2\le p]\]
has the same law as a chordal restriction sample of parameter $\tilde{\xi}(\beta_1,...,\beta_p)$.

\subsection{Radial conformal restriction property}
\begin{definition}\label{def::collection_radial}
Let $\LA_r$ be the collection of all compact subset $A\subset\overline{\U}$ such that
\[0\not\in A, 1\not\in A,\quad A=\overline{A\cap\U},\quad \text{ and }\U\setminus A\text{ is simply connected}.\]
Denote by $\Phi_A$ the conformal map from $\U\setminus A$ onto $\U$ such that\footnote{Riemann's Mapping Theorem asserts that, if we have one interior point and one boundary point, there exists a unique conformal map from $\U\setminus A$ onto $\U$ that fixes the interior point and the boundary point.}
\[\Phi_A(0)=0,\quad \Phi_A(1)=1.\]
\end{definition}
We are interested in closed random subset $K$ of $\bar{\U}$ such that
\begin{enumerate}
\item [(1)] $K\cap\partial\U=\{1\}$, $0\in K$, $K$ is connected and $\U\setminus K$ is connected
\item [(2)] For any $A\in\LA_r$, the law of $\Phi_A(K)$ conditioned on $[K\cap A=\emptyset]$ is the same as $K$.
\end{enumerate}
The combination of the above properties is called \textbf{radial conformal restriction property}, and the law of such a random set is called \textbf{radial restriction measure}. It turns out there exists only a two-parameter family $\QQ(\alpha,\beta)$ of such probability measures (more detail in Sections~\ref{sec::radial_sle} and \ref{sec::radial_restriction}) for
\[\beta\ge 5/8, \quad \alpha\le \xi(\beta).\]
The radial restriction measures are closely related to whole-plane intersection exponent (will be proved in Section~\ref{subsec::restriction_exponents_radial}):
Suppose $K_1,...,K_p$ are $p$ independent radial restriction samples of parameters $(\xi(\beta_1),\beta_1),...,(\xi(\beta_p),\beta_p)$ respectively. The ``fill-in" of the union of these sets
\[\bigcup_{j=1}^p K_j\] conditioned on the event (viewed as a limit)
\[[K_{j_1}\cap K_{j_2}=\emptyset, 1\le j_1<j_2\le p]\]
has the same law as a radial restriction sample of parameter $(\xi(\beta_1,...,\beta_p), \tilde{\xi}(\beta_1,...,\beta_p))$.

%% file: tex/brownian_motion_excursion_loop.tex
\subsection{Brownian motion}
Suppose that $W^1,W^2$ are two independent 1-dimensional BMs, then $B:=W^1+iW^2$ is a complex BM.
\begin{lemma}
Suppose $B$ is a complex BM and $u$ is a harmonic function, then $u(B)$ is a local martingale.
\end{lemma}
\begin{proof}
By It\^{o}'s Formula,
\begin{align*}
du(B_t)&=\partial_xu(B_t)dW_t^1+\partial_yu(B_t)dW_t^2+\frac{1}{2}(\partial_{xx}u(B_t)+\partial_{yy}u(B_t))dt\\
&=\partial_xu(B_t)dW_t^1+\partial_yu(B_t)dW_t^2.
\end{align*}
\end{proof}
\begin{proposition}\label{prop::bm_conformalinvariance}
Suppose $D$ is a domain and $f:D\to\C$ is a conformal map. Let $B$ be a complex BM starting from $z\in D$, stopped at
\[\tau_D:=\inf\{t\ge 0: B_t\not\in D\}.\]
Then the time-changed process $f(B)$ has the same law as a complex BM starting from $f(z)$ stopped at $\tau_{f(D)}$. Namely,
define
\[S(t)=\int_0^t |f'(B_u)|^2du,\quad 0\le t<\tau_D,\]
\[\sigma(s)=S^{-1}(s), \quad i.e. \int_0^{\sigma(s)}|f'(B_u)|^2du=s.\]
Then $(Y_s=f(B_{\sigma(s)}),0\le s\le S_{\tau_D})$ has the same law as BM starting from $f(z)$ stopped at $\tau_{f(D)}$.
\end{proposition}
\begin{proof}
Write $f=u+iv$ where $u,v$ are harmonic and \[\partial_xu=\partial_yv, \quad \partial_yu=-\partial_xv.\]
We have \[du(B_t)=\partial_xu(B_t)dW_t^1+\partial_yu(B_t)dW_t^2,\quad dv(B_t)=\partial_xv(B_t)dW_t^1+\partial_yv(B_t)dW_t^2.\]
Thus the two coordinates of $f(B)$ are local martingales and the quadratic variation is
\[\langle u(B)\rangle_t=\langle v(B)\rangle_t=\int_0^t (\partial_xu^2(B_s)+\partial_yu^2(B_s))ds=\int_0^t |f'(B_s)|^2ds,\]
\[\langle u(B), v(B)\rangle_t=(\partial_xu(B_t)\partial_xv(B_t)+\partial_yu(B_t)\partial_yv(B_t))dt=0.\]
Thus the two coordinates of $Y$ are independent local martingales with quadratic variation $t$ which implies that $Y$ is a complex BM.
\end{proof}

We introduce some notations about measures on continuous curves. 
Let $\LK$ be the set of all parameterized continuous planar curves $\gamma$ defined on a time interval $[0,t_{\gamma}]$. $\LK$ can be viewed as a metric space
\[d_{\LK}(\gamma,\eta)=\inf_{\theta}\sup_{0\le s\le t_{\gamma}} |s-\theta(s)|+|\gamma(s)-\eta(\theta(s))|\]
where the inf is taken over all increasing homeomorphisms $\theta: [0,t_{\gamma}]\to [0,t_{\eta}]$. Note that $\LK$ under this metric does not identify curves that are the same modulo time-reparametrization.

If $\mu$ is any measure on $\LK$, let $|\mu|=\mu(\LK)$ denote the total mass. If $0<|\mu|<\infty$, let $\mu^{\sharp}=\mu/|\mu|$ be $\mu$ normalized to be a probability measure. Let $M$ denote the set of finite Borel measures on $\LK$. This is a metric space under Prohorov metric \cite[Section 6]{BillingsleyConvergenceProbabilityMeasures}. To show that a sequence of finite measures $\mu_n$ converges to a finite measure $\mu$, it suffices to show that
\[|\mu_n|\to |\mu|,\quad \mu_n^{\sharp}\to \mu^{\sharp}.\]

If $D$ is a domain, we say that $\gamma$ is in $D$ if $\gamma(0,t_{\gamma})\subset D$, and let $\LK(D)$ be the set of $\gamma\in\LK$ that are in $D$. Note that, we do not require the endpoints of $\gamma$ to be in $D$. Suppose $f:D\to D'$ is a conformal map and $\gamma\in\LK(D)$. Let
\begin{equation}\label{eqn::bm_timechange}S(t)=\int_0^{t}|f'(\gamma(s))|^2ds.\end{equation}
If $S(t)<\infty$ for all $t<t_{\gamma}$, define $f\circ\gamma$ by
\[(f\circ\gamma)(S(t))=f(\gamma(t)).\]
If $\mu$ is a measure supported on the set of $\gamma$ in $\LK(D)$ such that $f\circ\gamma$ is well-defined and in $\LK(D')$, then $f\circ\mu$ denotes the measure
\[f\circ\mu(V)=\mu[\gamma: f\circ\gamma\in V].\]

\medbreak
\noindent\textbf{From interior point to interior point}
\medbreak
Let $\mu(z,\cdot;t)$ denote the law of complex BM $(B_s,0\le s\le t)$ starting from $z$. We can write
\[\mu(z,\cdot;t)=\int_{\C}\mu(z,w;t)dA(w)\]
where $dA(w)$ denotes the area measure and $\mu(z,w;t)$ is a measure on continuous curve from $z$ to $w$. The total mass of $\mu(z,w;t)$ is
\begin{equation}\label{eqn::bm_heatkernel}|\mu(z,w;t)|=\frac{1}{2\pi t}\exp(-\frac{|z-w|^2}{2t}).\end{equation}
The normalized measure $\mu^{\sharp}(z,w;t)=\mu(z,w;t)/|\mu(z,w;t)|$ is a probability measure, and it is called a \textbf{Brownian bridge} from $z$ to $w$ in time $t$.
The total mass $|\mu(z,w;t)|$ is also called heat kernel and Equation~(\ref{eqn::bm_heatkernel}) can be obtained through
\[|\mu(z,D; t)|=\PP^z[B_t\in D]=\int_D \frac{1}{2\pi t}\exp(-\frac{|z-w|^2}{2t})dA(w).\]

The measure $\mu(z,w)$ is defined by
\[\mu(z,w)=\int_0^{\infty}\mu(z,w;t)dt.\]
This is a $\sigma$-finite infinite measure. If $D$ is a domain and $z,w\in D$, define $\mu_D(z,w)$ to be $\mu(z,w)$ restricted to curves stayed in $D$. If $z\neq w$, and $D$ is a domain such a BM in $D$ eventually exits $D$, then $|\mu_D(z,w)|<\infty$. Define \textbf{Green's function}
\[G_D(z,w)=\pi |\mu_D(z,w)|.\]
In particular, $G_{\U}(0,z)=-\log|z|$.
\begin{proposition}[Conformal Invariance] Suppose $f: D\to D'$ is a conformal map, $z,w$ are two interior points in $D$. Then
\[f\circ\mu_D(z,w)=\mu_{f(D)}(f(z), f(w)).\]
In particular,
\[G_{f(D)}(f(z),f(w))=G_D(z,w),\quad (f\circ\mu_D)^{\sharp}(z,w)=\mu^{\sharp}_{f(D)}(f(z),f(w)).\]
\end{proposition}
\begin{proof}
By Proposition~\ref{prop::bm_conformalinvariance}.
\end{proof}
 
\medbreak
\noindent\textbf{From interior point to boundary point}
\medbreak
Suppose $D$ is a connected domain. Let $B$ be a BM starting from $z\in D$ and stopped at
\[\tau_D=\inf\{t: B_t\not\in D\}.\]
Define $\mu_D(z,\partial D)$ to be the law of $(B_s,0\le s\le \tau_D)$. If $D$ has nice boundary (i.e. $\partial D$ is piecewise analytic), we can write
\[\mu_D(z,\partial D)=\int_{\partial D}\mu_D(z,w)dw\]
where $dw$ is the length measure and $\mu_D(z,w)$ is a measure on continuous curves from $z$ to $w$. Define \textbf{Poisson's kernel}
\[H_D(z,w)=|\mu_D(z,w)|.\]
In particular, $H_{\U}(0,w)=1/(2\pi)$. The measure $ |\mu_D(z,w)|dw$ on $\partial D$ is called the harmonic measure seen from $z$, and the Poisson's kernel is the density of this harmonic measure.  

The normalized measure $\mu^{\sharp}_D(z,w)=\mu_D(z,w)/|\mu_D(z,w)|$ can also be viewed as the law of BM conditioned to exit $D$ at $w$ when $w$ is a nice boundary point (i.e. $\partial D$ is analytic in a neighborhood of $w$):
\begin{eqnarray*}
\lefteqn{\PP^z[\cdot\cond B_{\tau_D}\in \U(w,\eps)]}\\
&=&\frac{\mu_D(z,\U(w,\eps))[\cdot]}{|\mu_D(z,\U(w,\eps))|}\\
&=&\frac{\int_{\U(w,\eps)}\mu_D(z,u)[\cdot]du}{\int_{\U(w,\eps)}|\mu_D(z,u)|du}\rightarrow \mu_D^{\sharp}(z,w),\quad\text{as }\eps\to 0.
\end{eqnarray*}

\begin{proposition}[Conformal Covariance] Suppose $D$ is a connected domain with nice boundary, $z\in D$,$ w\in\partial D$ is a nice boundary point. Let $f:D\to D'$ be a conformal map. Then
\[f\circ\mu_D(z,w)=|f'(w)|\mu_{f(D)}(f(z),f(w)).\]
In particular,
\[|f'(w)|H_{f(D)}(f(z),f(w))=H_D(z,w),\quad (f\circ\mu_D)^{\sharp}(z,w)=\mu^{\sharp}_{f(D)}(f(z),f(w)).\]
\end{proposition}
\noindent\textbf{Relation between the two}
\begin{proposition}\label{prop::relation_two}
Suppose $D$ is a connected domain with nice boundary, $z\in D$, and $w\in\partial D$ is a nice boundary point. Let $\bn_w$ denote the inward normal at $w$, then
\[\lim_{\eps\to 0}\frac{1}{2\eps}\mu_D(z,w+\eps\bn_w)=\mu_D(z,w).\]
In particular,
\[\frac{1}{2\pi\eps}G_D(z,w+\eps\bn_w)\to H_D(z,w), \text{ as }\eps\to 0;\]
\[\mu^{\sharp}_D(z,w_n)\to \mu_D^{\sharp}(z,w)\text{ as }w_n\in D \to w.\]
\end{proposition}
\begin{proof}
Note that $G_{\U}(0,z)=-\log|z|$ and $H_{\U}(0,w)=1/2\pi$, thus 
\[G_{\U}(0,(1-\eps)w)=-\log(1-\eps)\approx \eps=2\pi\eps H_{\U}(0,w).\]
This implies that 
\[\lim_{\eps\to 0}\frac{1}{2\eps}\mu_{\U}(0, (1-\eps)w)=\mu_{\U}(0,w).\]
The conclusion for general domain $D$ can be obtained via conformal invariance/covariance. 
\end{proof}
\subsection{Brownian excursion}
Suppose $D$ is a connected domain with nice boundary and $z,w$ are two distinct nice boundary points. Define the measure on Brownian path from $z$ to $w$ in $D$:
\[\mu_D(z,w)=\lim_{\eps\to 0}\frac{1}{\eps}\mu_D(z+\eps\bn_z,w)=\lim_{\eps\to 0}\frac{1}{2\eps^2}\mu_D(z+\eps\bn_z,w+\eps\bn_w).\]
Denote
\[H_D(z,w)=|\mu_D(z,w)|.\]
The normalized measure $\mu_D^{\sharp}(z,w)$ is called Brownian excursion measure in $D$ with two end points $z,w\in\partial D$.
Note that
\[H_D(z,w)=\lim_{\eps\to 0}H_D(z+\eps\bn_z,w),\quad H_{\HH}(0,x)=\frac{1}{\pi x^2}.\]
\begin{proposition}[Conformal Covariance] Suppose that $f:D\to D'$ is a conformal map, and $z,w\in\partial D$, $f(z),f(w)\in\partial f(D)$ are nice boundary points. Then
\[f\circ\mu_D(z,w)=|f'(z)f'(w)|\mu_{f(D)}(f(z),f(w)).\]
In particular,
\[|f'(z)f'(w)|H_{f(D)}(f(z),f(w))=H_D(z,w),\quad (f\circ\mu_D)^{\sharp}(z,w)=\mu_{f(D)}^{\sharp}(f(z),f(w)).\]
\end{proposition}
The following proposition is an equivalent expression of the conformal restriction property of Brownian excursion we discussed in Subsection~\ref{subsec::intro_bm_be}.
\begin{proposition}\label{prop::be_restriction}
Suppose $A\in\LA_c$ and $\Phi_A$ is the conformal map defined in Definition~\ref{def::collection_chordal}. Let $e$ be a Brownian excursion whose law is $\mu^{\sharp}_{\HH}(0,\infty)$. Then
\[\PP[e\cap A=\emptyset]=\Phi'_A(0).\]
\end{proposition}
\begin{proof}
Although $\mu_{\HH}(0,\infty)$ has zero total mass, the normalized measure can still be defined through the limit procedure:
\[\mu_{\HH}^{\sharp}(0,\infty)=\lim_{x\to\infty}\mu_{\HH}^{\sharp}(0,x)=\lim_{x\to\infty}\mu_{\HH}(0,x)/|\mu_{\HH}(0,x)|.\]
Thus
\begin{align*}
\PP[e\cap A=\emptyset]&=\lim_{x\to\infty}\mu_{\HH}(0,x)[e\cap A=\emptyset]/|\mu_{\HH}(0,x)|\\
&=\lim_{x\to\infty}|\mu_{\HH\setminus A}(0,x)|/|\mu_{\HH}(0,x)|\\
&=\lim_{x\to\infty}H_{\HH\setminus A}(0,x)/H_{\HH}(0,x)\\
&=\lim_{x\to\infty}\Phi'_A(0)\Phi'_A(x)H_{\HH}(0,\Phi_A(x))/H_{\HH}(0,x)=\Phi'_A(0).
\end{align*}
\end{proof}
Note that, the excursion measure $\mu^{\sharp}_{\HH}(0,\infty)$ introduced in this section does coincide with the one we introduced in Section \ref{subsec::intro_bm_be}: by Proposition \ref{prop::relation_two} and the continuous dependence of Brownian bridge measure on the end points, we have 
\[\mu_{\HH}^{\sharp}(0,\infty)=\lim_{\eps\to 0, R\to\infty}\mu^{\sharp}_{\HH}(\eps i, R+i\R).\]
\begin{corollary}
Suppose $e_1,...,e_n$ are $n$ independent Brownian excursion with law $\mu^{\sharp}_{\HH}(0,\infty)$, denote
$\Sigma=\cup_{j=1}^n e_j$, then for any $A\in\LA_c$,
\[\PP[\Sigma\cap A=\emptyset]=\Phi'_A(0)^n.\]
\end{corollary}
\begin{corollary}\label{cor::be_conformalrestriction}
Let $e$ be a Brownian excursion with law $\mu^{\sharp}_{\HH}(x,y)$ where $x,y\in\R, x\neq y$. Then, for any closed subset $A\subset\bar{\HH}$ such that $x,y\not\in A$ and $\HH\setminus A$ is simply connected, we have that
\[\PP[e\cap A=\emptyset]=\Phi'(x)\Phi'(y)\]
where $\Phi$ is any conformal map from $\HH\setminus A$ onto $\HH$ that fixes $x$ and $y$. Note that the quantity $\Phi'(x)\Phi'(y)$ is unique although $\Phi$ is not unique.
\end{corollary}

\begin{definition}
Suppose $D$ has nice boundary, then \textbf{Brownian excursion measure} is defined as
\[\mu_{D,\partial D}^{exc}=\int_{\partial D}\int_{\partial D}\mu_D(z,w)dzdw.\]
Generally, if $I$ is a subsect of $\partial D$, define
\[\mu_{D,I}^{exc}=\int_I\int_I\mu_D(z,w)dzdw.\]
\end{definition}
\begin{proposition}[Conformal Invariance] Suppose $D,D'$ have nice boundaries and $f:D\to D'$ is a conformal map. Then
\[f\circ\mu_{D,I}^{exc}=\mu_{f(D),f(I)}^{exc},\quad f\circ\mu_{D,\partial D}^{exc}=\mu_{f(D),\partial f(D)}^{exc}.\]
\end{proposition}
\begin{proof}
\begin{align*}
f\circ\mu_{D,I}^{exc}&=\int_I\int_I f\circ\mu_D(z,w)dzdw\\
&=\int_I\int_I |f'(z)f'(w)|\mu_{f(D)}(f(z),f(w))dzdw\\
&=\int_{f(I)}\int_{f(I)}\mu_{f(D)}(z,w)dzdw=\mu_{f(D),f(I)}^{exc}.\end{align*}
\end{proof}
\begin{theorem}\label{thm::ppp_be_restriction}
Let $(e_j,j\in J)$ be a Poisson point process with intensity $\pi\beta\mu^{exc}_{\HH,\R_-}$ for some $\beta>0$. Set $\Sigma=\cup_j e_j$. For any $A\in\LA_c$ such that $A\cap \R\subset (0,\infty)$, we have that
\[\PP[\Sigma\cap A=\emptyset]=\Phi'_A(0)^{\beta}.\]
\end{theorem}
\begin{figure}[ht!]
\begin{center}
\includegraphics[width=0.7\textwidth]{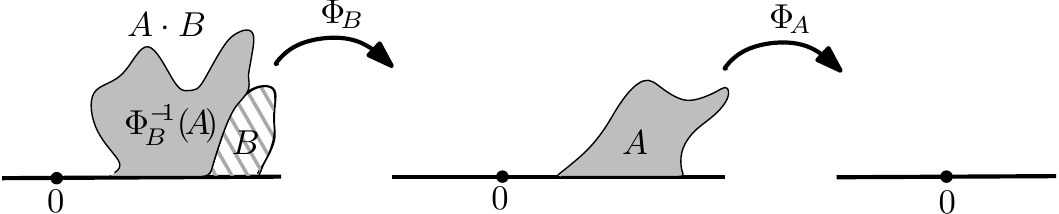}
\end{center}
\caption{\label{fig::phiAcdotB} The set $A\cdot B$ is the union of $\Phi_B^{-1}(A)$ and $B$.}
\end{figure}
\begin{proof}
Denote by $N_A$ the number of excursions in $(e_j, j\in J)$ that intersect $A$, then we see that $\{\Sigma\cap A=\emptyset\}$ is equivalent to $\{N_A=0\}$ where $N_A$ has the law of Poisson distribution with parameter $\pi\beta\mu^{exc}_{\HH,\R_-}[e\cap A\neq\emptyset])$. Thus 
\[\PP[\Sigma\cap A=\emptyset]=\exp(-\pi\beta\mu^{exc}_{\HH,\R_-}[e\cap A\neq\emptyset]).\]
We only need to show that
\[\mu^{exc}_{\HH,\R_-}[e\cap A\neq\emptyset]=-\frac{1}{\pi}\log\Phi'_A(0).\]
This will be obtained by two steps: First, there exists a constant $c$ such that
\begin{equation}\label{eqn::ppp_be_first}\mu^{exc}_{\HH,\R_-}[e\cap A\neq\emptyset]=c\log\Phi'_A(0).\end{equation}
Second,
\begin{equation}\label{eqn::ppp_be_second}c=-1/\pi.\end{equation}
For the first step, we need to introduce a set $A\cdot B$: Suppose $A,B\in\LA_c$ such that $A\cap\R\subset(0,\infty)$ and $B\cap\R\subset(0,\infty)$. Define (see Figure~\ref{fig::phiAcdotB})
\[A\cdot B=\Phi_B^{-1}(A)\cup B.\]
Then clearly, $\Phi_{A\cdot B}=\Phi_A\circ\Phi_B$, and
\begin{equation}\label{eqn::ppp_be_linear1}\log\Phi'_{A\cdot B}(0)=\log\Phi'_A(0)+\log\Phi'_B(0).\end{equation}
For the Brownian excursion measure, we have
\begin{align*}
\mu^{exc}_{\HH,\R_-}[e\cap A\cdot B\neq\emptyset]
&=\mu^{exc}_{\HH,\R_-}[e\cap B\neq\emptyset]+\mu^{exc}_{\HH,\R_-}[e\cap B=\emptyset,e\cap A\cdot B\neq\emptyset]\\
&=\mu^{exc}_{\HH,\R_-}[e\cap B\neq\emptyset]+\mu^{exc}_{\HH\setminus B,\R_-}[e\cap\Phi_B^{-1}(A)\neq\emptyset]\\
&=\mu^{exc}_{\HH,\R_-}[e\cap B\neq\emptyset]+\mu^{exc}_{\HH,\R_-}[e\cap A\neq\emptyset]
\end{align*}
In short, we have
\[\mu^{exc}_{\HH,\R_-}[e\cap A\cdot B\neq\emptyset]=\mu^{exc}_{\HH,\R_-}[e\cap B\neq\emptyset]+\mu^{exc}_{\HH,\R_-}[e\cap A\neq\emptyset].\]
Combining with Equation (\ref{eqn::ppp_be_linear1}), we have Equation (\ref{eqn::ppp_be_first}).\footnote{Idea: $F(t+s)=F(t)+F(s)\rightsquigarrow F(t)=ct$. For precise proof, see \cite[Theorem 8]{WernerConformalRestrictionRelated}.}
Generally, if $I=[a,b]\subset\R_-$, we have
\begin{equation}\label{eqn::ppp_be_first_general}\mu_{\HH,I}^{exc}[e\cap A\neq\emptyset]=c\log(\Phi'_A(a)\Phi'_A(b)).\end{equation}
Next, we will find the constant. Suppose $I=[a,b]\subset\R_-$ and $A\in\LA_c$ such that $A\cap \R\subset (0,\infty)$.
\begin{align*}
\mu^{exc}_{\HH,I}[e\cap A\neq\emptyset]
&=\int_I\int_I\mu_{\HH}(x,y)[e\cap A\neq\emptyset]dxdy\\
&=\int_I\int_IH_{\HH}(x,y)\mu^{\sharp}_{\HH}(x,y)[e\cap A\neq\emptyset]dxdy\\
&=\int_I\int_IH_{\HH}(x,y)(1-\Phi'_{x,y}(x)\Phi'_{x,y}(y))dxdy, \tag{\text{By Corollary~\ref{cor::be_conformalrestriction}}}
\end{align*}
where $\Phi_{x,y}$ is any conformal map from $\HH\setminus A$ onto $\HH$ that fixes $x$ and $y$. Define the Mobius transformation
\[m(z)=\left(\frac{x-y}{\Phi_A(x)-\Phi_A(y)}\right)(z-x)+x,\]
then $\Phi_{x,y}=m\circ\Phi_A$ would do the work. Thus
\[\mu^{exc}_{\HH,I}[e\cap A\neq\emptyset]=\int_I\int_I\frac{1}{\pi|x-y|^2}\left(1-\left(\frac{x-y}{\Phi_A(x)-\Phi_A(y)}\right)^2\Phi_A'(x)\Phi'_A(y)\right)dxdy.\]
It is not clear to see how this double integral would give $c\log(\Phi'_A(a)\Phi'_A(b))$. However, we only need to decide the the constant $c$ which is much easier. Suppose $I=[-\eps,0]$, and set $a_1=\Phi'_A(0)$ and $a_2=\Phi''_A(0)/2$, we have that
\[\mu^{exc}_{\HH,I}[e\cap A\neq\emptyset]=\frac{a_2^2}{\pi a_1^2}\eps^2+o(\eps^2),\quad \log(\Phi'_A(0)\Phi'_A(-\eps))=-\frac{a_2^2}{a_1^2}\eps^2+o(\eps^2).\]
Combining these two expansions, we obtain that the constant $c=-1/\pi$.
\end{proof}
\subsection{Brownian loop}
Suppose $(\gamma(t), 0\le t\le t_{\gamma})\in\LK$ is a loop, i.e. $\gamma(0)=\gamma(t_{\gamma})$. Such a $\gamma$ can be considered as a function defined on $(-\infty,\infty)$ satisfying $\gamma(s)=\gamma(s+t_{\gamma})$ for any $s\in\R$. Let $\tilde{\LK}\subset\LK$ be the collection of such loops. Define, for $r\in\R$, the shift operator $\theta_r$ on loops:
\[\theta_r\gamma(s)=\gamma(r+s).\]
We say that two loops $\gamma,\gamma'$ are equivalent if for some $r$, we have $\gamma'=\theta_r\gamma$. Denote by $\tilde{\LK}_u$ the set of unrooted loops, i.e. the equivalent classes. We will define Brownian loop measure on unrooted loops.

Recall that $\mu(z,\cdot;t)$ denotes the law of complex BM $(B_s,0\le s\le t)$ and
\[\mu(z,\cdot;t)=\int\mu(z,w;t)dA(w).\]
Now we are interested in loops, i.e. $\mu(z,z;t)$ where the path starts from $z$ and returns back to $z$. We have that
\[|\mu(z,z;t)|=\frac{1}{2\pi t},\quad \mu(z,z)=\int_0^{\infty}\mu(z,z;t)dt=\int_0^{\infty}\frac{1}{2\pi t}\mu^{\sharp}(z,z;t)dt.\]
We define \textbf{Brownian loop measure} $\mu^{loop}$ by
\begin{equation}
\mu^{loop}=\int_{\C}\frac{1}{t_{\gamma}}\mu(z,z)dA(z)=\int_{\C}\int_0^{\infty}\frac{1}{2\pi t^2}\mu^{\sharp}(z,z;t)dtdA(z).
\end{equation}
The term $1/t_{\gamma}$ corresponds to averaging over the root and $\mu^{loop}$ is defined on unrooted loops.
If $D$ is a domain, define $\mu_D^{loop}$ to be $\mu^{loop}$ restricted to the curves totally contained in $D$.
\begin{proposition}[Conformal Invariance] If $f:D\to D'$ is a conformal map, then
\[f\circ\mu^{loop}_D=\mu^{loop}_{f(D)}.\]
\end{proposition}
\begin{proof}
We call a Borel measurable function $F:\tilde{\LK}\to [0,\infty)$ a unit weight if, for any $\gamma\in\tilde{\LK}$, we have
\[\int_0^{t_{\gamma}}F(\theta_r\gamma)=1.\]
One example is $F(\gamma)=1/t_{\gamma}$. For any unit weight $F$, since $\mu^{loop}$ is defined on unrooted loops, we have that
\begin{equation}\label{eqn::loopmeasure_unitweight}
\mu^{loop}=\int_{\C}F\mu(z,z)dA(z).
\end{equation}
Define a function $F_f$ on $\tilde{\LK}$ in the following way:  for any $\gamma\in\tilde{\LK}$,
\[F_f(\gamma)=|f'(\gamma(0))|^2/t_{f\circ\gamma}.\]
Recall the time change in Equation (\ref{eqn::bm_timechange}), we can see that $F_f$ is a unit weight:
\[\int_0^{t_{\gamma}}F_f(\theta_r\gamma)dr=\int_0^{t_{\gamma}}|f'(\gamma(r))|^2/t_{f\circ\gamma}dr=1.\]
Thus,
\[\mu^{loop}_D=\int_D F_f\mu_D(z,z)dA(z)=\int_D \frac{1}{t_{f\circ\gamma}}|f'(z)|^2\mu_D(z,z)dA(z).\]
Therefore, 
\begin{align*}
f\circ\mu_D^{loop}&=
\int_D\frac{1}{t_{f\circ\gamma}}|f'(z)|^2 f\circ\mu_D(z,z)dA(z)\\
&=\int_D\frac{1}{t_{f\circ\gamma}}|f'(z)|^2 \mu_{f(D)}(f(z),f(z))dA(z)\\
&=\int_{f(D)}\frac{1}{t_{\eta}}\mu_{f(D)}(w,w)dA(w)=\mu^{loop}_{f(D)}.
\end{align*}
\end{proof}

\begin{theorem}\label{thm::ppp_bl_restriction}
Denote by $\mu^{loop}_{\U,0}$ the measure $\mu^{loop}_{\U}$ restricted to the loops surrounding the origin.
Let $(l_j,j\in J)$ be a Poisson point process with intensity $\alpha\mu^{loop}_{\U,0}$ for some $\alpha>0$. Set $\Sigma=\cup_jl_j$. For any closed subset $A\subset\overline{\U}$ such that $0\not\in A$, $\U\setminus A$ is simply connected, we have that
\[\PP[\Sigma\cap A=\emptyset]=\Phi_A'(0)^{-\alpha},\]
where $\Phi_A$ is the conformal map from $\U\setminus A$ onto $\U$ with $\Phi_A(0)=0, \Phi_A'(0)>0$.
\end{theorem}
\begin{proof}
Since
\[\PP[\Sigma\cap A=\emptyset]=\exp(-\alpha\mu^{loop}_{\U,0}[l\cap A\neq\emptyset]),\]
we only need to show that
\[\mu^{loop}_{\U,0}[l\cap A\neq\emptyset]=\log\Phi_A'(0).\]
Similar as in the proof of Theorem \ref{thm::ppp_be_restriction}, this can be obtained by two steps: First, there exists a constant $c$ such that
\begin{equation}\label{eqn::ppp_bl_first}\mu^{loop}_{\U,0}[l\cap A\neq\emptyset]=c\log\Phi'_A(0).\end{equation}
Second,
\begin{equation}\label{eqn::ppp_bl_second}c=1.\end{equation}
For the first step, it can be proved in the similar way as the proof of the first step of Theorem~\ref{thm::ppp_be_restriction}, and the precise proof can be found in \cite[Lemma 4]{WernerSelfavoidingLoop}. But for the second step, it is more complicate. We omit this part and the interested readers can consult \cite{WernerSelfavoidingLoop, SheffieldWernerCLE, LawlerWernerBrownianLoopsoup}.
\end{proof}

%% file: tex/chordal_sle.tex
\subsection{Introduction}
Schramm Lowner Evolution (SLE for short) was introduced by Oded Schramm in 1999 \cite{SchrammFirstSLE} as the candidates of the scaling limits of discrete statistical physics models. We will take percolation as an example. Suppose $D$ is a domain and we have a discrete lattice of size $\eps$ inside $D$, say the triangular lattice $\eps\T\cap D$. The critical percolation on the discrete lattice is the following: At each vertex of the lattice, there is a random variable which is black or white with equal probability $1/2$. All these random variables are independent. We can see that there are interfaces separating black vertices from white vertices. To be precise, let us fix two distinct boundary points $a,b\in\partial D$. Denote by $\partial_L$ (resp. $\partial_R$) the part of the boundary from $a$ to $b$ clockwise (resp. counterclockwise). We fix all vertices on $\partial_L$ (resp. $\partial_R$) to be white (resp. black). And then sample independent black/white random variables at the vertices inside $D$. Then there exists a unique interface from $a$ to $b$ separating black vertices from white vertices (see Figure~\ref{fig::percolation}). We denote this interface by $\gamma^{\eps}$, and call it the critical percolation interface in $D$ from $a$ to $b$.
\begin{figure}[ht!]
\begin{center}
\includegraphics[width=0.5\textwidth]{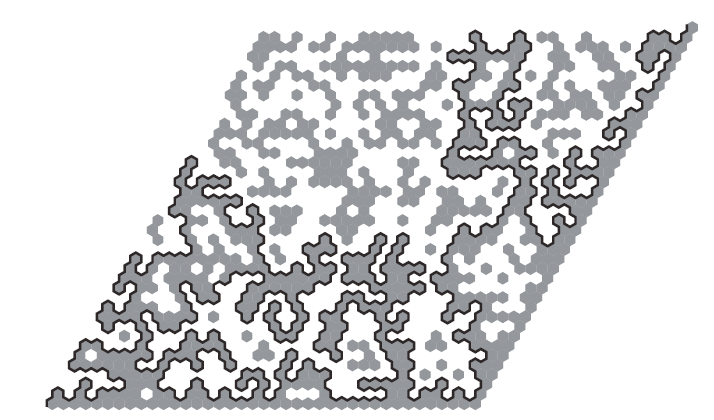}
\end{center}
\caption{\label{fig::percolation}There exists a unique interface from the left-bottom corner to right-top corner separating black vertices from white vertices. (Picture by Julien Dub\'edat, from \cite{WernerLecturePercolation})}
\end{figure}

It is worthwhile to point out the domain Markov property in this discrete model: Starting from $a$, we move along $\gamma^{\eps}$ and stopped at some point $\gamma^{\eps}(n)$. Given $L=(\gamma^{\eps}(1),...,\gamma^{\eps}(n))$, the future part of $\gamma^{\eps}$ has the same law as the critical percolation interface in $D\setminus L$ from $\gamma^{\eps}(n)$ to $b$.

People believe that the discrete interface $\gamma^{\eps}$ will converge to some continuous path in $D$ from $a$ to $b$ as $\eps$ goes to zero. Assume this is true and suppose $\gamma$ is the limit continuous curve in $D$ from $a$ to $b$. Then we would expect that the limit should satisfies the following two properties: Conformal Invariance and Domain Markov Property which is the continuous analog of discrete domain Markov property. SLE curves are introduced from this motivation: chordal SLE curves are random curves in simply connected domains connecting two boundary points such that they satisfy: (see Figure~\ref{fig::SLE_conf_inv_Markov})
\begin{itemize}
\item \textbf{Conformal Invariance}: $\gamma$ is an SLE curve in $D$ from $a$ to $b$, $\varphi$ is a conformal map, then $\varphi(\gamma)$ has the same law as an SLE curve in $\varphi(D)$ from $\varphi(a)$ to $\varphi(b)$.
\item \textbf{Domain Markov Property}: $\gamma$ is an SLE curve in $D$ from $a$ to $b$, given $\gamma([0,t])$, $\gamma([t,\infty))$ has the same law as an SLE curve in $D\setminus \gamma[0,t]$ from $\gamma(t)$ to $b$.
\end{itemize}
\begin{figure}[ht!]
\begin{subfigure}[b]{0.48\textwidth}
\begin{center}
\includegraphics[width=\textwidth]{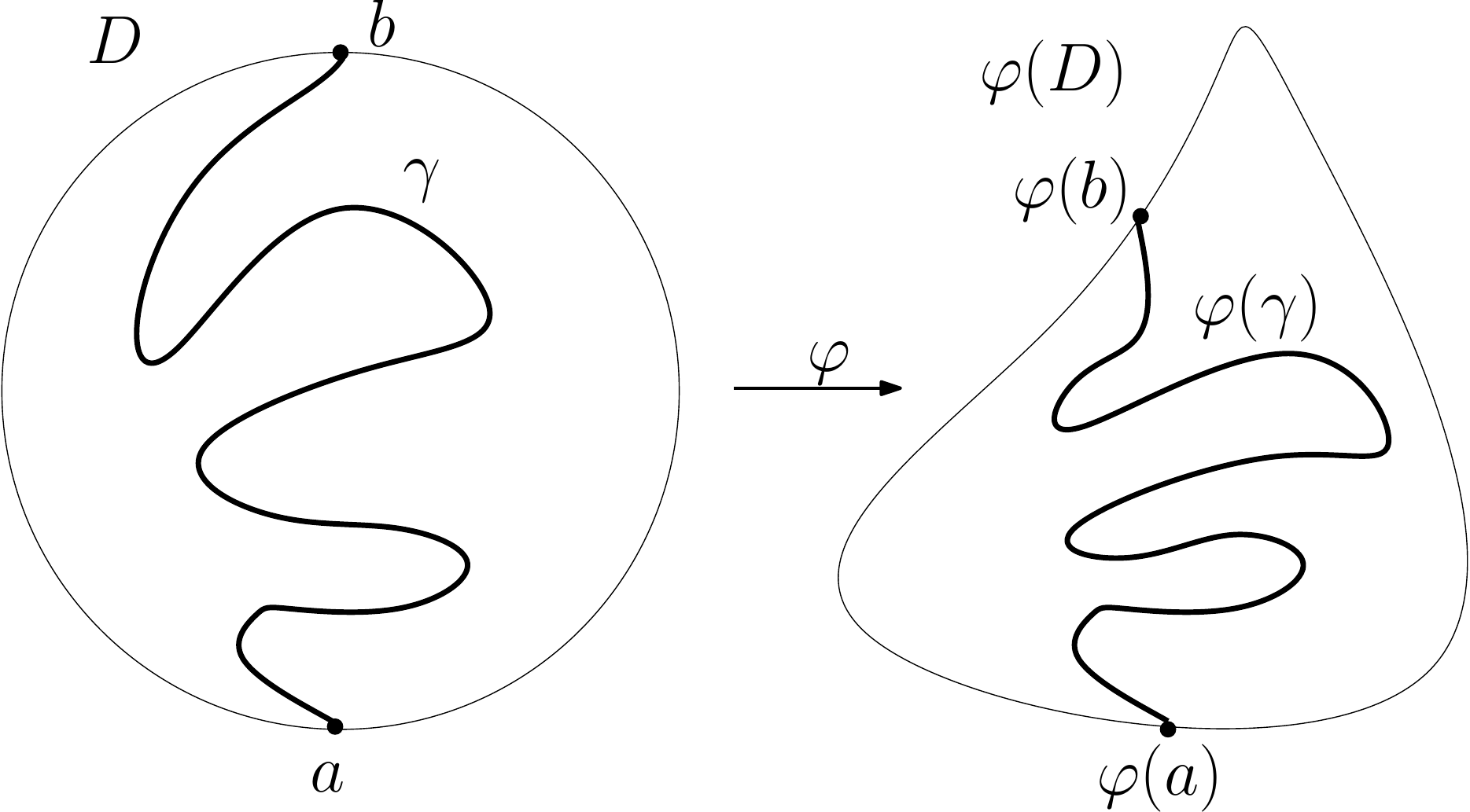}
\end{center}
\caption{Conformal Invariance.}
\end{subfigure}
$\quad$
\begin{subfigure}[b]{0.48\textwidth}
\begin{center}\includegraphics[width=\textwidth]{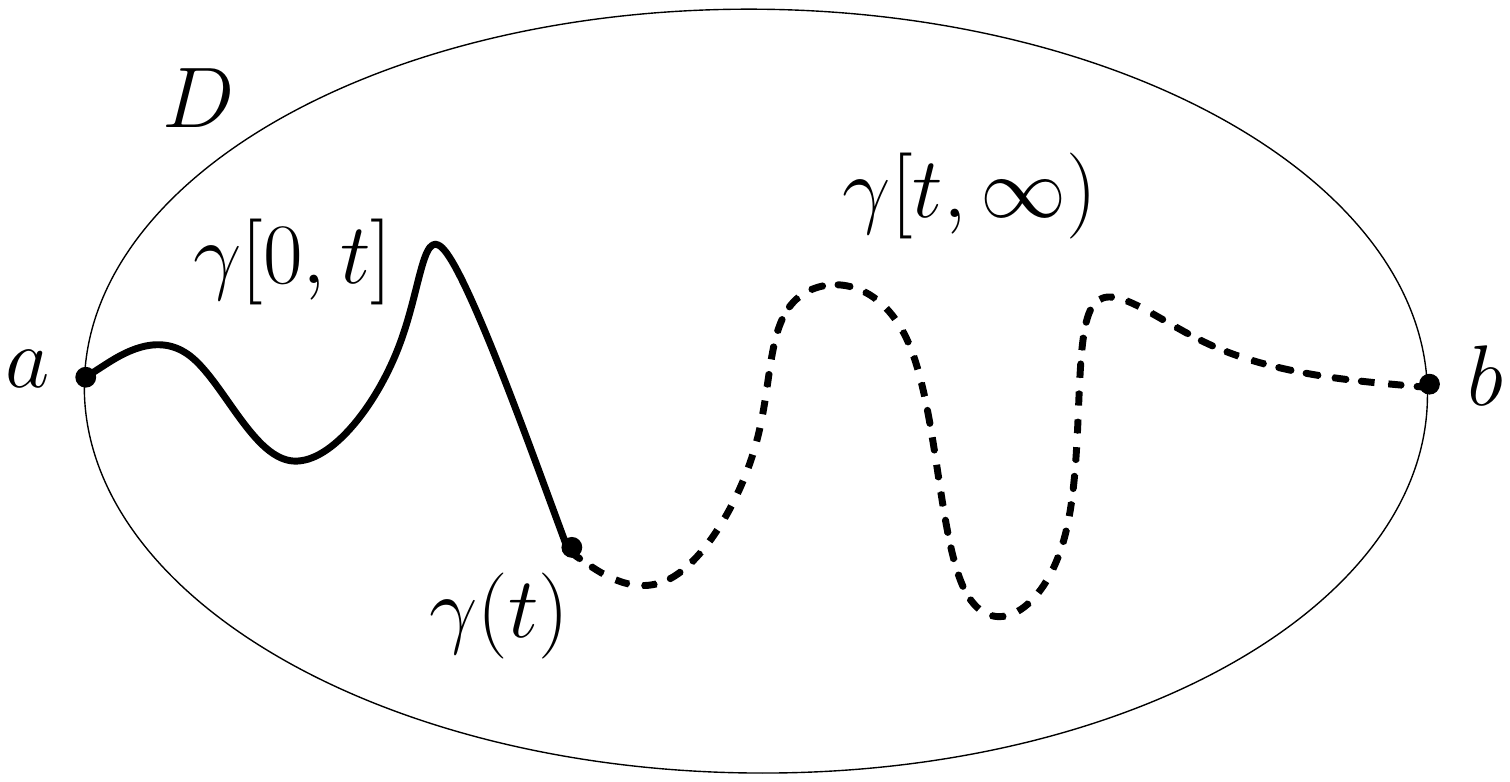}\end{center}
\caption{Domain Markov Property.}
\end{subfigure}
\caption{\label{fig::SLE_conf_inv_Markov}Characterization of SLE.}
\end{figure}

The following of section is organized as follows: In Subsection~\ref{subsec::loewnerchain}, we introduce one time parameterization of continuous curves, called Loewner chain, that is suitable to describe the domain Markov property of the curves. In Subsection~\ref{subsec::chordalsle}, we introduce the definition of chordal SLE and discuss its basic properties.
Without loss of generality, we choose to work in the upper half-plane $\HH$ and suppose the two boundary points are $0$ and $\infty$.
\subsection{Loewner chain}\label{subsec::loewnerchain}
\noindent\textbf{Half-plane capacity}
\medbreak
We call a compact subset $K$ of $\overline{\HH}$ a hull if $H=\HH\setminus K$ is simply connected. Riemann's mapping theorem asserts that there exists a conformal map $\Psi$ from $H$ onto $\HH$ that $\Psi(\infty)=\infty$. In fact, if $\Psi$ is such a map, then $c\Psi+c'$ for $c>0,c'\in\R$ is also a map from $H$ onto $\infty$ fixing $\infty$. We choose to fix the two-degree freedom in the following way. The map $\Psi$ can be expanded near $\infty$: there exist $b_1,b_0,b_{-1},...$
\[\Psi(z)=b_1z+b_0+\frac{b_{-1}}{z}+\cdots+\frac{b_{-n}}{z^n}+o(z^{-n}), \quad \text{as }z\to\infty.\]
Furthermore, since $\Psi$ preserves the real axis near $\infty$, all coefficients $b_j$ are real. Hence, for each $K$, there exists a unique conformal map $\Psi$ from $H=\HH\setminus K$ onto $\HH$ such that
\[\Psi(z)=z+0+O(1/z),\quad\text{as }z\to\infty.\]
We call such a conformal map the conformal map from $H=\HH\setminus K$ onto $\HH$ normalized at $\infty$, and denote it by $\Psi_K$.
In particular, there exists a real $a=a(K)$ such that
\[\Psi(z)=z+2a/z+o(1/z),\quad\text{as }z\to\infty.\]
We also denote $a(K)$ by $a(\Psi_{K})$. This number $a(K)$ can be viewed as the size of $K$:
\begin{lemma}
The quantity $a(K)$ is a non-negative increasing function of the set $K$.
\end{lemma}
\begin{proof}
We first show that $a$ is non-negative. Suppose that $Z=X+iY$ is a complex BM starting from $Z_0=iy$ for some $y>0$ large (so that $iy\in H=\HH\setminus K$) and stopped at its first exit time $\tau$ from $H$. Let $\Psi$ be the conformal map from $H$ onto $\HH$ normalized at the infinity, then $\Im(\Psi(z)-z)$ is a bounded harmonic function in $H$. The martingale stopping theorem therefore shows that
\[\E[\Im(\Psi(Z_{\tau}))-Y_{\tau}]=\Im(\Psi(iy)-iy)=-\frac{2a}{y}+o(\frac{1}{y}),\quad \text{as }y\to\infty.\]
Since $\Psi(Z_{\tau})$ is real, we have that
\[2a=\lim_{y\to\infty}y\E[Y_{\tau}]\ge 0.\]
Next we show that $a$ is increasing. Suppose $K,K'$ are hulls and $K\subset K'$. Let $\Psi_1=\Psi_K$, and let $\Psi_2$ be the conformal map from $\HH\setminus\Psi_K(K'\setminus K)$ onto $\HH$ normalized at infinity. Then $\Psi_{K'}=\Psi_2\circ\Psi_1$, and
\[a(K')=a(K)+a(\Psi_2)\ge a(K).\]
\end{proof}
We call $a(K)$ the capacity of $K$ in $\HH$ seen from $\infty$ or \textbf{half-plane capacity}. 
Here are several simple facts:
\begin{itemize}
\item When $K$ is vertical slit $[0,iy]$, we have $\Psi_K(z)=\sqrt{z^2+y^2}$. In particular, $a(K)=y^2/4$.
\item If $\lambda>0$, then $a(\lambda K)=\lambda^2 a(K)$
\end{itemize}
\medbreak
\noindent\textbf{Loewner chain}
\medbreak
Suppose that $(W_t,t\ge 0)$ is a continuous real function with $W_0=0$. For each $z\in\overline{\HH}$, define the function $g_t(z)$ as the solution to the ODE
\[\partial_t g_t(z)=\frac{2}{g_t(z)-W_t},\quad g_0(z)=z.\]
This is well-defined as long as $g_t(z)-W_t$ does not hit 0. Define
\[T(z)=\sup\{t>0: \min_{s\in[0,t]}|g_s(z)-W_s|>0\}.\]
This is the largest time up to which $g_t(z)$ is well-defined. Set
\[K_t=\{z\in\overline{\HH}: T(z)\le t\},\quad H_t=\HH\setminus K_t.\]
We can check that $g_t$ is a conformal map from $H_t$ onto $\HH$ normalized at $\infty$. For each $t$, we have
\[g_t(z)=z+2t/z+o(1/z),\quad \text{as } z\to\infty.\]
In other words, $a(K_t)=t$.
The family $(K_t,t\ge 0)$ is called the \textbf{Loewner chain} driven by $(W_t,t\ge 0)$.
\subsection{Chordal SLE}\label{subsec::chordalsle}
\noindent\textbf{Definition}
\medbreak
\noindent Chordal SLE$_{\kappa}$ for $\kappa\ge 0$ is the Loewner chain driven by $W_t=\sqrt{\kappa}B_t$ where $B$ is a 1-dimensional BM with $B_0=0$. 
\begin{lemma}  Chordal SLE$_{\kappa}$ is scale-invariant. 
\end{lemma}
\begin{proof}
Since $W$ is scale-invariant, i.e. for any $\lambda>0$, the process $W^{\lambda}_t=W_{\lambda t}/\sqrt{\lambda}$ has the same law as $W$. Set $g^{\lambda}_t(z)=g_{\lambda t}(\sqrt{\lambda}z)/\sqrt{\lambda}$, we have
\[\partial_t g^{\lambda}_t(z)=\frac{2}{g^{\lambda}_t-W^{\lambda}_t},\quad g^{\lambda}_0(z)=z.\]
Thus $(K_{\lambda t}/\sqrt{\lambda},t\ge 0)$ has the same law as $K$.
\end{proof}
For general simply connected domain $D$ with two boundary points $x$ and $y$, we define SLE$_\kappa$ in $D$ from $x$ to $y$ to be the image of chordal SLE$_{\kappa}$ in $\HH$ from $0$ to $\infty$ under any conformal map from $\HH$ to $D$ sending the pair $0,\infty$ to $x,y$. Since SLE$_{\kappa}$ is scale-invariant, the SLE in $D$ from $x$ to $y$ is well-defined. 
\begin{lemma} Chordal SLE$_{\kappa}$ satisfies domain Markov property.
Moreover, the law of SLE$_{\kappa}$ is symmetric with respect to the imaginary axis.
\end{lemma}
\begin{proof}
Proof of domain Markov property: Since BM is a strong Markov process with independent increments, for any stopping time $T$, the process
$(g_{T}(K_{t+T}\setminus K_T)-W_T,t\ge 0)$ is independent of $(K_s,0\le s \le T)$ and has the same law as $K$. Thus, for any stopping time $T$ and given $K_T$, the conditional law of $(K_{t+T}, t\ge 0)$ is the same as an SLE in $\HH\setminus K_T$.

Proof of symmetry: Suppose that $K$ is a Loewner chain driven by $W$. Let $\tilde{K}$ be the image of $K$ under the reflection with respect to the imaginary axis. Define $(\tilde{g}_t)_{t\ge 0}$ to be the corresponding sequence of conformal maps for $\tilde{K}$. Then we could check that $\tilde{K}$ is a Loewner chain driven by $-W$. Since $-W$ has the same as $W$, we know that $\tilde{K}$ has the same law as $K$. This implies that the law of SLE$_{\kappa}$ is symmetric with respect to the imaginary axis.
\end{proof}
\begin{proposition}
For all $\kappa\in [0,4]$, chordal SLE$_{\kappa}$ is almost surely a simple continuous curve, i.e. there exists a simple continuous curve $\gamma$ such that
$K_t=\gamma[0,t]$ for all $t\ge 0$. See Figure \ref{fig::sle_curve_explanation}. Moreover, almost surely, we have $\lim_{t\to\infty}\gamma(t)=\infty$.
\end{proposition}
\noindent The proof of this proposition is difficult, we will omit it in the lecture. The interested readers could consult \cite{RohdeSchrammSLEBasicProperty}.
\begin{figure}[ht!]
\begin{center}
\includegraphics[width=0.47\textwidth]{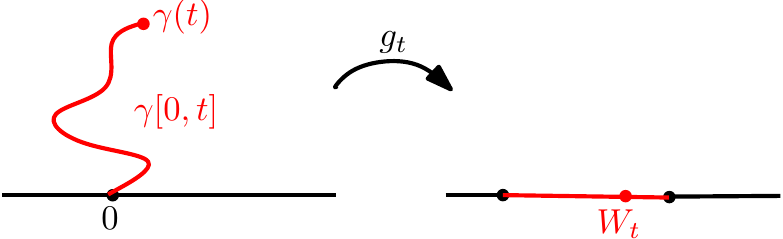}
\end{center}
\caption{\label{fig::sle_curve_explanation} The map $g_t$ is the conformal map from $\HH\setminus \gamma[0,t]$ onto $\HH$ normalized at $\infty$. And the tip of the curve $\gamma(t)$ is the preimage of $W_t$ under $g_t$: $\gamma(t)=g_t^{-1}(W_t)$.}
\end{figure}
\medbreak
\noindent\textbf{Restriction property of SLE$_{8/3}$}
\medbreak
In this part, we will compute the probability of SLE$_{8/3}$ process  $\gamma$ to avoid a set $A\in\LA_c$. To this end, we need to analyze the behavior of the image $\tilde{\gamma}=\Phi_A(\gamma)$. Define  $T=\inf\{t: \gamma(t)\in A\}$, and for $t<T$, set 
\[\tilde{\gamma}[0,t]:=\Phi_A(\gamma[0,t]).\]
Recall that $\Phi_A$ is the conformal map from $\HH\setminus A$ onto $\HH$ with $\Phi_A(0)=0$, $\Phi_A(\infty)=\infty$, and $\Phi_A(z)/z\to 1$ as $z\to\infty$, and that $g_t$ is the conformal map from $\HH\setminus \gamma[0,t]$ onto $\HH$ normalized at infinity. Define $\tilde{g}_t$ to be the conformal map from $\HH\setminus\tilde{\gamma}[0,t]$ onto $\HH$ normalized at infinity and $h_t$ the conformal map from $\HH\setminus g_t(A)$ onto $\HH$ such that Equation (\ref{eqn::sle8over3_exchange_relation}) holds. See Figure~\ref{fig::sle8over3_exchange}.
\begin{equation}\label{eqn::sle8over3_exchange_relation}
h_t\circ g_t=\tilde{g}_t\circ\Phi_A.
\end{equation}
\begin{figure}[ht!]
\begin{center}
\includegraphics[width=0.47\textwidth]{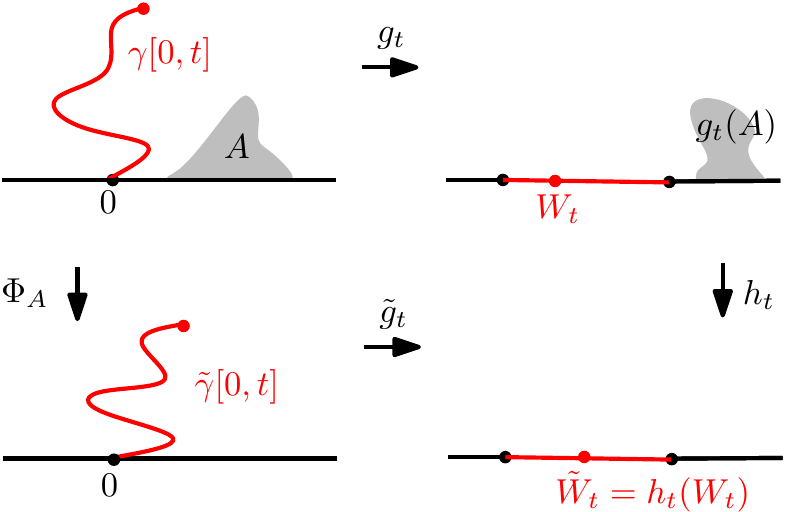}
\end{center}
\caption{\label{fig::sle8over3_exchange} The map $\Phi_A$ is the conformal map from $\HH\setminus A$ onto $\HH$ with $\Phi_A(0)=0$, $\Phi_A(\infty)=\infty$, and $\Phi_A(z)/z\to 1$ as $z\to\infty$. The map $g_t$ is the conformal map from $\HH\setminus \gamma[0,t]$ onto $\HH$ normalized at infinity. The map $\tilde{g}_t$ is the conformal map from $\HH\setminus\tilde{\gamma}[0,t]$ onto $\HH$ normalized at infinity. The map $h_t$ is the conformal map from $\HH\setminus g_t(A)$ onto $\HH$ such that $h_t\circ g_t=\tilde{g}_t\circ\Phi_A$.}
\end{figure}
\begin{proposition}\label{prop::chordalsle8over3_localmart}
When $\kappa=8/3$, the process
\[M_t=h'_t(W_t)^{5/8},\quad t<T\]
is a local martingale.
\end{proposition}
\begin{proof}
Define
\[a(t)=a(\gamma[0,t]\cup A)=a(A)+a(\tilde{\gamma}[0,t]).\]
Let $t(u)$ be the inverse of $a$: for any $u>0$, define $t(u)=\inf\{t: a(t)=u\}$. Note that $a(t(u))=u$. In other words, the curve $(\tilde{\gamma}(t(u)), u>0)$ is parameterized by half-plane capacity. Therefore, 
\[\partial_u \tilde{g}_{t(u)}(z)=\frac{2}{\tilde{g}_{t(u)}(z)-\tilde{W}_{t(u)}}.\]
Since $\partial_t a \partial_u t=1$, we have
\[\partial_t\tilde{g}_t(z)=\frac{2\partial_t a}{\tilde{g}_t(z)-\tilde{W}_t}.\]
Plugging Equation (\ref{eqn::sle8over3_exchange_relation}), we have that
\begin{equation}\label{eqn::sle8over3_first}
(\partial_th_t)(z)+h'_t(z)\frac{2}{z-W_t}=\frac{2\partial_t a}{h_t(z)-h_t(W_t)}.
\end{equation}
We can first find $\partial_t a$: multiply $h_t(z)-h_t(W_t)$ to both sides of Equation (\ref{eqn::sle8over3_first}), and then let $z\to W_t$, we have
\[\partial_t a=h'_t(W_t)^2.\]
Then Equation (\ref{eqn::sle8over3_first}) becomes
\begin{equation}\label{eqn::sle8over3_second}
(\partial_th_t)(z)=\frac{2h_t'(W_t)^2}{h_t(z)-h_t(W_t)}-\frac{2h'_t(z)}{z-W_t}.
\end{equation}
Differentiate Equation (\ref{eqn::sle8over3_second}) with respect to $z$, we have
\[(\partial_th_t)'(z)=\frac{-2h_t'(W_t)^2h'_t(z)}{(h_t(z)-h_t(W_t))^2}+\frac{2h'_t(z)}{(z-W_t)^2}-\frac{2h''_t(z)}{z-W_t}.\]
Let $z\to W_t$, we have
\[(\partial_th'_t)(W_t)=h''_t(W_t)dW_t+\left(\frac{h''_t(W_t)^2}{2h'_t(W_t)}+(\frac{\kappa}{2}-\frac{4}{3})h'''_t(W_t)\right)dt.\]
When $\kappa=8/3$,
\[d(h'_t(W_t))^{5/8}=\frac{5h''_t(W_t)}{8h'_t(W_t)^{3/8}}dW_t.\]
\end{proof}
\begin{theorem}\label{thm::sle8over3_restriction}
Suppose $\gamma$ is a chordal SLE$_{8/3}$ in $\HH$ from 0 to $\infty$. For any $A\in\LA_c$, we have
\[\PP[\gamma\cap A=\emptyset]=\Phi_A'(0)^{5/8}.\]
\end{theorem}
\begin{proof}
Since we can approximate any compact hull by compact hulls with smooth boundary, we may assume that $A$ has smooth boundary. Set
\[M_t=(h'_t(W_t))^{5/8}.\]
If $e$ is a Brownian excursion with law $\mu^{\sharp}_{\HH}(W_t,\infty)$, then $h'_t(W_t)$ is the probability of $e$ to avoid $g_t(A)$. See Proposition \ref{prop::be_restriction}. Thus, for $t<T$, we have $h'_t(W_t)\le 1$ and $M$ is bounded.

If $T=\infty$, we have $\lim_{t\to\infty}h'_t(W_t)=1$.

If $T<\infty$, we have $\lim_{t\to T}h'_t(W_t)=0$.

Roughly speaking, when $T=\infty$, $g_t(A)$ will be far away from $W_t$ as $t\to\infty$ and thus the probability for $e$ to avoid $g_t(A)$ converges to 1; whereas, when $T<\infty$, the set $g_t(A)$ will be very close to $W_t$ as $t\to T$ and the probability for $e$ to avoid $g_t(A)$ converges to 0. (See \cite{LawlerSchrammWernerConformalRestriction} for details.)

Since $M$ converges in $L^1$ and a.s. when $t\to T$, we have that
\[\PP[\gamma\cap A=\emptyset]=\PP[T=\infty]=\E[M_T]=\E[M_0]=\Phi_A'(0)^{5/8}.\]
\end{proof}

%% file: tex/chordal_restriction.tex
\subsection{Setup for chordal restriction sample}\label{subsec::setup_chordal_restriction}
Let $\Omega$ be the collection of closed sets $K$ of $\overline{\HH}$ such that
\[K\cap \R=\{0\}, K\text{ is unbounded, } K \text{ is connected}\]
\[\text{and } \HH\setminus K \text{ has two connected components.}\]
Recall that $\LA_c$ is defined in Definition~\ref{def::collection_chordal}. We endow $\Omega$ with the $\sigma$-field generated by the events $[K\in\Omega: K\cap A=\emptyset]$ where $A\in\LA_c$. This family of events is closed under finite intersection, so that a probability measure on $\Omega$ is characterized by the values of $\PP[K\cap A=\emptyset]$ for $A\in\LA_c$: Let $\PP,\PP'$ are two probability measures on $\Omega$. If $\PP[K\cap A=\emptyset]=\PP'[K\cap A=\emptyset]$ for all $A\in\LA_c$, then $\PP=\PP'$.
\begin{definition} A probability measure $\PP$ on $\Omega$ is said to satisfy chordal conformal restriction property, if the following is true:
\begin{enumerate}
\item [(1)] For any $\lambda>0$, $\lambda K$ has the same law as $K$;
\item [(2)] For any $A\in\LA_c$, $\Phi_A(K)$ conditioned on $[K\cap A=\emptyset]$ has the same law as $K$.
\end{enumerate}
\end{definition}
\begin{theorem}\label{thm::chordal_restriction}
Chordal restriction measures have the following description.
\begin{enumerate}
\item [(1)] (Characterization) A chordal restriction measure is fully characterized by a positive real $\beta>0$ such that, for every $A\in\LA_c$,
\begin{equation}\label{eqn::chordal_restiction_characterization}
\PP[K\cap A=\emptyset]=\Phi'_A(0)^{\beta}.
\end{equation}
We denote the corresponding chordal restriction measure by $\PP(\beta)$.
\item [(2)] (Existence) The measure $\PP(\beta)$ exists if and only if $\beta\ge 5/8$.
\end{enumerate}
\end{theorem}
\begin{remark}
We already know that $\PP(\beta)$ exist for $\beta=1$ (by Proposition \ref{prop::be_restriction}), $\beta=5/8$ (by Theorem \ref{thm::sle8over3_restriction}), and $\beta=5/8m+n$ for $m\ge 1, n\ge 1$.
\end{remark}

\begin{proof}[Proof of Theorem \ref{thm::chordal_restriction}. Characterization]
Suppose that $K$ is scale-invariant and satisfies Equation (\ref{eqn::chordal_restiction_characterization}) for every $A\in\LA_c$, then we could check that $K$ does satisfy chordal conformal restriction property. Thus we only need to show that chordal restriction measures have only one degree of freedom.

Fix $x\in\R\setminus\{0\}$ and let $\eps>0$. We claim that the probability
\[\PP[K\cap B(x,\eps)\neq\emptyset]\]
decays like $\eps^2$ as $\eps$ goes to zero, and the limit
\[\lim_{\eps\to 0}\frac{1}{\eps^2}\PP[K\cap B(x,\eps)\neq\emptyset]\]
exists which we denote by $\lambda(x)$ (The detail of the proof of this argument could be found in \cite{WuConformalRestrictionRadial}).

Furthermore, $\lambda(x)\in (0,\infty)$. Since $K$ is scale-invariant, we have that, for any $y>0$,
\[\lambda(yx)=\lim_{\eps\to 0}\frac{1}{\eps^2}\PP[K\cap B(yx,\eps)\neq\emptyset]=\lim_{\eps\to 0}\frac{1}{\eps^2}\PP[K\cap B(x,\eps/y)\neq\emptyset]=y^{-2}\lambda(x).\]
Since $\lambda$ is an even function, there exists $c >0$ such that
\[\lambda(x)=c x^{-2}.\]
Since there is only one-degree of freedom, when $K$ satisfies chordal restriction property, we must have that Equation (\ref{eqn::chordal_restiction_characterization}) holds for some $\beta>0$.

Denote $f_{x,\eps}=\Phi_{\bar{\U}(x,\eps)}$. In fact,
\[f_{x,\eps}(z)=z+\frac{\eps^2}{z-x}+\frac{\eps^2}{x}.\]
Note that,
\[\PP[K\cap \U(x,\eps)\neq\emptyset]\approx \lambda(x)\eps^2,\]
and that
\[1-f_{x,\eps}'(0)^{\beta}\approx \beta\frac{\eps^2}{x^2}.\]
This implies that $\beta=c$.
\end{proof}

In the following of this section, we will first show that $\PP(\beta)$ does not exist for $\beta<5/8$ and then construct all $\PP(\beta)$ for $\beta> 5/8$.

\subsection{Chordal SLE$_{\kappa}(\rho)$ process}
\noindent\textbf{Definition}
\medbreak
Suppose $\kappa>0,\rho>-2$. Chordal SLE$_{\kappa}(\rho)$ process is the Loewner chain driven by $W$ which is the solution to the following SDE:
\begin{equation}\label{eqn::slekapparho_sde}
dW_t=\sqrt{\kappa}dB_t+\frac{\rho dt}{W_t-O_t},\quad dO_t=\frac{2dt}{O_t-W_t}, \quad W_0=O_0=0, \quad O_t\le W_t.\end{equation}
The evolution is well-defined at times when $W_t>O_t$, but a bit delicate when $W_t=O_t$. We first show the existence of the solution to this SDE.

Define $Z_t$ to be the solution to the Bessel equation
\[dZ_t=\sqrt{\kappa}dB_t+(\rho+2)\frac{dt}{Z_t},\quad Z_0=0.\] In other words, $Z$ is $\sqrt{\kappa}$ times a Bessel process of dimension
\[d=1+2(\rho+2)/\kappa.\] 
This process is well-defined for all $\rho>-2$, and for all $t\ge 0$,
\[\int_0^{t}\frac{du}{Z_u}=(Z_t-\sqrt{\kappa}B_t)/(\rho+2)<\infty. \] Then define
\[O_t=-2\int_0^{t}\frac{du}{Z_u},\quad W_t=Z_t+O_t.\]
Clearly, $(W_t,O_t)$ is a solution to Equation (\ref{eqn::slekapparho_sde}). When $\rho=0$, we get the ordinary SLE$_{\kappa}$.

Second, we explain the geometric meaning of the process $(O_t,W_t)$. Recall
\[\partial_t g_t(z)=\frac{2}{g_t(z)-W_t},\quad g_0(z)=z.\]
Suppose $(K_t,t\ge 0)$ is the Loewner chain generated by $W$, then $g_t$ is the conformal map from $\HH\setminus K_t$ onto $\HH$ normalized at $\infty$. The point $W_t$ is the image of the tip, and $O_t$ is the image of the leftmost point of $\R\cap K_t$. See Figure~\ref{fig::slekapparho_explanation}. Basic properties of SLE$_{\kappa}(\rho)$ process: Fix $\kappa\in [0,4]$, $\rho>-2$,
\begin{itemize}
\item It is scale-invariant: for any $\lambda>0$, $(\lambda^{-1}K_{\lambda^2t},t\ge 0)$ has the same law as $K$.
\item $(K_t,t\ge 0)$ is generated by a continuous curve $(\gamma(t),t\ge 0)$ in $\overline{\HH}$ from 0 to $\infty$.
\item If $\rho\ge \kappa/2-2$, the dimension of the Bessel process $Z_t=W_t-O_t$ is greater than $2$ and $Z$ does not hit zero, thus almost surely $\gamma\cap\R=\{0\}$. If $\rho\in(-2,\kappa/2-2)$, almost surely $\gamma\cap\R\neq\{0\}$ and $K_{\infty}\cap\R=(-\infty,0]$.\footnote{When $\rho>0$, the process $W_t$ gets a push away from $O_t$, the curve is repelled from $\R_-$. When $\rho<0$, the curve is attracted to $\R_-$. When $\rho<\kappa/2-2$, the attraction is strong enough so that the curve touches $\R_-$.}
\end{itemize}
\begin{figure}[ht!]
\begin{subfigure}[b]{0.47\textwidth}
\begin{center}
\includegraphics[width=\textwidth]{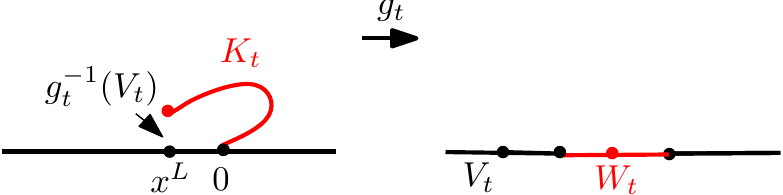}
\end{center}
\caption{When $\rho\ge \kappa/2-2$, the curve does not hit $\R_-$. }
\end{subfigure}
$\quad$
\begin{subfigure}[b]{0.47\textwidth}
\begin{center}\includegraphics[width=\textwidth]{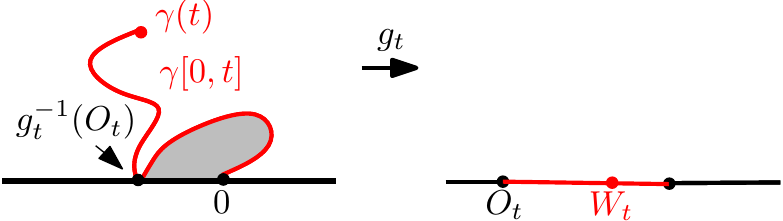}
\end{center}
\caption{When $\rho\in (-2,\kappa/2-2)$, the curve touches the boundary. }
\end{subfigure}
\caption{\label{fig::slekapparho_explanation}Geometric meaning of $(O_t,W_t)$ in SLE$_{\kappa}(\rho)$ process. The preimage of $W_t$ under $g_t$ is the tip of the curve, the preimage of $O_t$ under $g_t$ is the leftmost point of $K_t\cap\R$. }
\end{figure}

\begin{theorem}\label{thm::sle8over3rho_rightrestriction}
Fix $\rho>-2$. Let $(K_t,t\ge 0)$ be the hulls of chordal SLE$_{8/3}(\rho)$ and $K=\cup_{t\ge 0}K_t$. Then $K$ satisfies the right-sided restriction property with exponent
\begin{equation}\label{eqn::beta_rho_relation}\beta=\frac{3\rho^2+16\rho+20}{32}.\end{equation}
In other words, for every $A\in\LA_c$ such that $A\cap\R\subset(0,\infty)$, we have
\[\PP[K\cap A=\emptyset]=\Phi_A'(0)^{\beta}.\]
\end{theorem}
\begin{proof}
The definitions of $g_t,\tilde{g}_t,h_t$ are recalled in Figure \ref{fig::sle8over3rho_exchange}. Set $T=\inf\{t: K_t\cap A\neq\emptyset\}$, and define, for $t<T$,
\[M_t=h_t'(W_t)^{5/8}h_t'(O_t)^{\rho(3\rho+4)/32}\left(\frac{h_t(W_t)-h_t(O_t)}{W_t-O_t}\right)^{3\rho/8}.\]
Then $(M_t,t<T)$ is a local martingale \cite[Lemma 8.9]{LawlerSchrammWernerConformalRestriction}:
\begin{equation*}
\begin{split}
dh_t(W_t)&=\left(\frac{\rho h_t'(W_t)}{W_t-O_t}-(5/3)h_t''(W_t)\right)dt+\sqrt{8/3}h_t'(W_t)dB_t,\\
dh_t'(W_t)&=\left(\frac{\rho h_t''(W_t)}{W_t-O_t}+\frac{h_t''(W_t)^2}{2h_t'(W_t)}\right)dt+\sqrt{8/3}h_t''(W_t)dB_t,\\
dh_t(O_t)&=\frac{2h_t'(W_t)^2}{h_t(O_t)-h_t(W_t)}dt,\\
dh_t'(O_t)&=\left(\frac{2h_t'(O_t)}{(O_t-W_t)^2}-\frac{2h_t'(W_t)^2h_t'(O_t)}{(h_t(O_t)-h_t(W_t))^2}\right)dt.
\end{split}\end{equation*}
Combining these identities, we see that $M$ is a local martingale.

Since $h_t'$ is decreasing in $(-\infty,W_t]$, we have
\[0\le h_t'(W_t)\le \frac{h_t(W_t)-h_t(O_t)}{W_t-O_t}\le h_t'(O_t)\le 1.\]

In fact, there exists $\delta>0$ such that $M_t\le h_t'(W_t)^{\delta}$. (We omit the proof of this point, details could be found in \cite[Lemma 8.10]{LawlerSchrammWernerConformalRestriction}). In particular, we have $M_t\le 1$ and $(M_t,t<T)$ is a bounded martingale.

If $T=\infty$, we have 
\[\lim_{t\to\infty}h'_t(W_t)=1\quad \text{and }  \lim_{t\to\infty}M_t=1.\]
If $T<\infty$, we have 
\[\lim_{t\to T}h'_t(W_t)=0, \quad \text{and } \lim_{t\to T}M_t=0.\]
Thus
\[\PP[K\cap A=\emptyset]=\PP[T=\infty]=\E[M_T]=M_0=\Phi_A'(0)^{\beta}\]
where $\beta$ is the same as in  Equation (\ref{eqn::beta_rho_relation}):
\[\beta=\frac{5}{8}+\frac{\rho(3\rho+4)}{32}+\frac{3\rho}{8}=\frac{3\rho^2+16\rho+20}{32}.\]
\end{proof}
\begin{figure}[ht!]
\begin{center}
\includegraphics[width=0.47\textwidth]{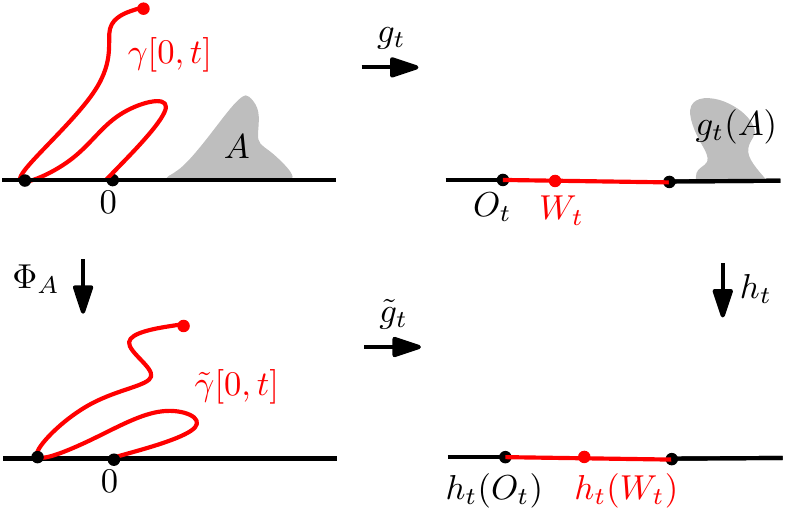}
\end{center}
\caption{\label{fig::sle8over3rho_exchange} The map $\Phi_A$ is the conformal map from $\HH\setminus A$ onto $\HH$ with $\Phi_A(0)=0$, $\Phi_A(\infty)=\infty$, and $\Phi_A(z)/z\to 1$ as $z\to\infty$. The map $g_t$ is the conformal map from $\HH\setminus \gamma[0,t]$ onto $\HH$ normalized at infinity. $\tilde{g}_t$ is the conformal map from $\HH\setminus\tilde{\gamma}[0,t]$ onto $\HH$ normalized at infinity. The map $h_t$ is the conformal map from $\HH\setminus g_t(A)$ onto $\HH$ such that $h_t\circ g_t=\tilde{g}_t\circ\Phi_A$.}
\end{figure}
\medbreak
\noindent\textbf{Setup for right-sided restriction property}
\medbreak
Let $\Omega^+$ be the collection of closed sets $K$ of $\bar{\HH}$ such that
\[K\cap \R=(-\infty,0],  K \text{ is connected and } \HH\setminus K \text{ is connected.}\]
Recall $\LA_c$ in Definition~\ref{def::collection_chordal}. Let $\LA_c^+$ denote the set of $A\in\LA_c$ such that $A\cap\R\subset(0,\infty)$. We endow $\Omega^+$ with the $\sigma$-field generated by the events $[K\in\Omega^+: K\cap A=\emptyset]$ where $A\in\LA_c^+$.
\begin{definition}
A probability measure $\PP$ on $\Omega^+$ is said to satisfy right-sided restriction property, if the following is true.
\begin{enumerate}
\item [(1)] For any $\lambda>0$, $\lambda K$ has the same law as $K$;
\item [(2)] For any $A\in\LA_c^+$, $\Phi_A(K)$ conditioned on $[K\cap A=\emptyset]$ has the same law as $K$.
\end{enumerate}\end{definition}
Similar to the proof of Theorem~\ref{thm::chordal_restriction}, we know that, if $\PP$ satisfies the right-sided restriction property, then there exists $\beta>0$ such that
\[\PP[K\cap A=\emptyset]=\Phi_A'(0)^{\beta},\quad\text{for all }A\in\LA_c^+.\]
\begin{remark}\label{rem::sle8over3rho_rightsided_restriction}
Theorem~\ref{thm::sle8over3rho_rightrestriction} states that SLE$_{8/3}(\rho)$ has the same law as the right boundary of the right-sided restriction sample with exponent $\beta$ which is related to $\rho$ through Equation (\ref{eqn::beta_rho_relation}). Note that
when $\rho$ spans $(-2,\infty)$, the quantity $\beta$ spans $(0,\infty)$. We could solve $\rho$ in terms of $\beta$ through Equation (\ref{eqn::beta_rho_relation}):
\begin{equation}\label{eqn::rho_beta_relation}\rho=\rho(\beta)=\frac{1}{3}(-8+2\sqrt{24\beta+1}).
\end{equation}
In particular, Theorem~\ref{thm::sle8over3rho_rightrestriction} also states the existence of right-sided restriction measure for all $\beta>0$.
\end{remark}
\begin{remark}
If $\beta\ge 5/8$, the right boundary of (two-sided) restriction measure $\PP(\beta)$ has the same law as SLE$_{8/3}(\rho)$ where $\rho=\rho(\beta)$ is given through Equation (\ref{eqn::rho_beta_relation}).
In particular, the right boundary of a Brownian excursion has the law of SLE$_{8/3}(2/3)$, the right boundary of the union of two independent Brownian excursions has the law of SLE$_{8/3}(2)$.
\end{remark}
\begin{remark} Recall Theorem~\ref{thm::ppp_be_restriction}, suppose $(e_j,j\in J)$ is a Poisson point process with intensity $\pi\beta\mu^{exc}_{\HH,\R_-}$, and set $\Sigma=\cup_je_j$, then the right boundary of $\Sigma$ has the same law as chordal SLE$_{8/3}(\rho)$ where $\rho=\rho(\beta)$ given by Equation (\ref{eqn::rho_beta_relation}) for all $\beta>0$.
\end{remark}
\begin{proof}[Proof of Theorem~\ref{thm::chordal_restriction}, $\PP(\beta)$ does not exist for $\beta<5/8$]
We prove by contradiction. Assume that the two-sided chordal restriction measure $\PP(\beta)$ exists for some $\beta<5/8$. Then the right boundary $\gamma$ of $K$ is SLE$_{8/3}(\rho)$ for $\rho=\rho(\beta)<0$ by Remark \ref{rem::sle8over3rho_rightsided_restriction}. 

On the one hand, the two-sided chordal restriction sample $K$ is symmetric with respect to the imaginary axis, thus the probability of $i$ staying to the right of $\gamma$ is less than $1/2$.
On the other hand, since $\rho<0$, the probability of $i$ staying to the right of $\gamma$ is strictly larger than the probability of $i$ staying to the right of SLE$_{8/3}$ which equals $1/2$, since SLE$_{8/3}$ is symmetric with respect to the imaginary axis and it is a simple continuous curve.
These two facts give us a contradiction.
\end{proof}
\subsection{Construction of $\PP(\beta)$ for $\beta>5/8$}
In the previous definition of SLE$_{\kappa}(\rho)$ process, there is a repulsion (when $\rho>0$) or attraction (when $\rho<0$) from $\R_-$. We will denote this process by SLE$^L_{\kappa}(\rho)$. Symmetrically, we denote by SLE$^R_{\kappa}(\rho)$ the same process only except that the repulsion or attraction is from $\R_+$. Namely, SLE$^R_{\kappa}(\rho)$ is the Loewner chain driven by $W$ which is the solution to the following SDE:
\begin{equation}
dW_t=\sqrt{\kappa}dB_t+\frac{\rho dt}{W_t-O_t},\quad dO_t=\frac{2dt}{O_t-W_t}, \quad W_0=O_0=0, \quad O_t\ge W_t.\end{equation}
Please compare it with Equation (\ref{eqn::slekapparho_sde}) and note that the only difference is $O_t\ge W_t$.  The process 
SLE$^R_{\kappa}(\rho)$ can be viewed as the image of SLE$^L_{\kappa}(\rho)$ under the reflection with respective to the imaginary axis.

From Theorem~\ref{thm::sle8over3rho_rightrestriction}, we know that SLE$^L_{8/3}(\rho)$ satisfies right-sided restriction property, thus similarly SLE$^R_{8/3}(\rho)$ satisfies left-sided restriction property. The idea to construct $K$ whose law is $\PP(\beta)$ for $\beta>5/8$ is the following: we first run an SLE$^L_{8/3}(\rho)$ as the right-boundary of $K$, and then given the right boundary, we run the left boundary according to the conditional law.

\begin{proposition}
Fix $\beta>5/8$, and $\rho=\rho(\beta)>0$ where $\rho(\beta)$ is given by Equation (\ref{eqn::rho_beta_relation}). Suppose $\gamma^R$ is a chordal SLE$^L_{8/3}(\rho)$ process in $\overline{\HH}$ from $0$ to $\infty$. Given $\gamma^R$, in the left-connected component of $\HH\setminus\gamma^R$, sample an SLE$^R_{8/3}(\rho-2)$ from $0$ to $\infty$ which is denoted by $\gamma^L$. Let $K$ be the closure of the union of the domains between $\gamma^L$ and $\gamma^R$. Then $K$ has the law of $\PP(\beta)$.
\end{proposition}
\begin{proof}
We only need to check, for all $A\in\LA_c$,
\[\PP[K\cap A=\emptyset]=\Phi_A'(0)^{\beta}.\]
Since $\gamma^R$ is an SLE$_{8/3}^L(\rho)$ process and it satisfies the right-sided restriction property by Theorem \ref{thm::sle8over3rho_rightrestriction}, we know that this is true for $A\in\LA_c^+$. We only need to prove it for $A\in\LA_c$ such that $A\cap\R\subset (-\infty,0)$. Let $(g_t,t\ge 0)$ be the solution of the Loewner chain for the process $\gamma^R$ and $(O_t,W_t,t\ge 0)$ be the solution of the SDE (\ref{eqn::slekapparho_sde}). Set $T=\inf\{t: \gamma^R(t)\in A\}$. For $t<T$, let $h_t$ be the conformal map from $\HH\setminus g_t(A)$ onto $\HH$ normalized at $\infty$. See Figure \ref{fig::chordalrestriction_construction}. Recall that
\[M_t=h_t'(W_t)^{5/8}h_t'(O_t)^{\rho(3\rho+4)/32}\left(\frac{h_t(W_t)-h_t(O_t)}{W_t-O_t}\right)^{3\rho/8}\]
is a local martingale, and that, since $h_t'$ is increasing on $[O_t, \infty)$, 
\[0\le h_t'(O_t)\le\frac{h_t(W_t)-h_t(O_t)}{W_t-O_t}\le h_t'(W_t)\le 1.\]
Since $\rho>0$, we have that $M_t\le h_t'(W_t)^{\beta}\le 1$ and thus $M$ is a bounded martingale.

If $T<\infty$, then \[h_t'(W_t)\to 0, \quad \text{and } M_t\to 0 \text{ as } t\to T.\]

If $T=\infty$, then \[h_t'(W_t)\to 1,\quad \frac{h_t(W_t)-h_t(O_t)}{W_t-O_t}\to 1,\]
and we have (apply Theorem \ref{thm::sle8over3rho_rightrestriction} to SLE$_{8/3}^R(\rho-2)$)
\[h_t'(O_t)^{\rho(3\rho+4)/32}\to \PP[\gamma^L\cap A=\emptyset\cond \gamma^R]\text{ as }t\to\infty.\]

Thus,
\[\PP[K\cap A=\emptyset]=E[1_{T=\infty}\E[1_{K\cap A=\emptyset}\cond \gamma^R]]=E[M_T]=M_0.\]
\end{proof}
\begin{figure}[ht!]
\begin{center}
\includegraphics[width=0.47\textwidth]{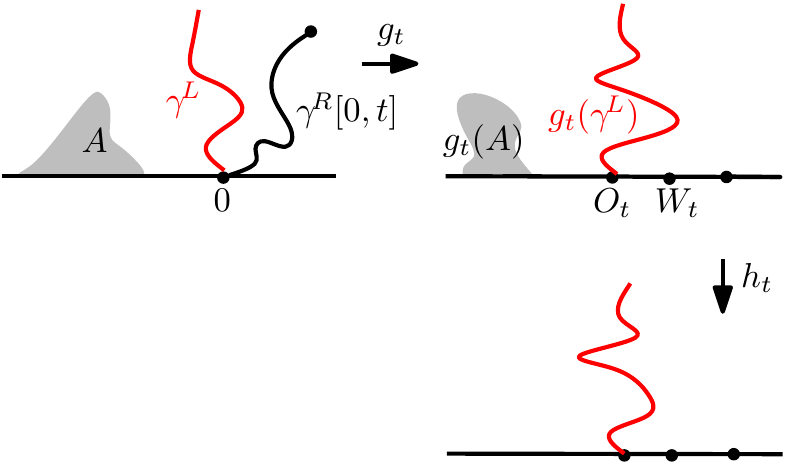}
\end{center}
\caption{\label{fig::chordalrestriction_construction} The map $g_t$ is the conformal map from $\HH\setminus\gamma^R[0,t]$ onto $\HH$ normalized at $\infty$. The map $h_t$ is the conformal map from $\HH\setminus g_t(A)$ onto $\HH$ normalized at $\infty$.}
\end{figure}

\subsection{Half-plane intersection exponents $\tilde{\xi}(\beta_1,...,\beta_p)$}\label{subsec::restriction_exponents_chordal}
Recall that
\[\tilde{\xi}(\beta_1,...,\beta_p)=\frac{1}{24}\left((\sqrt{24\beta_1+1}+\cdots+\sqrt{24\beta_p+1}-(p-1))^2-1\right),\]
and define
\begin{equation}\label{eqn::hatxi}
\hat{\xi}(\beta_1,...,\beta_p)=\tilde{\xi}(\beta_1,...,\beta_p)-\beta_1-\cdots-\beta_p.
\end{equation}
We will see in Proposition \ref{prop::disjoint_exponent} that, the function $\hat{\xi}$ is the exponent for the restriction samples to avoid each other, or the exponent for ``non-intersection". 

For $x\in\C$ and a subset $K\subset\C$, denote
\[x+K=\{x+z:z\in K\}.\]
\begin{proposition}\label{prop::disjoint_exponent}
Suppose $K_1,...,K_p$ are $p$ independent chordal restriction samples with exponents $\beta_1,...,\beta_p\ge 5/8$ respectively. Fix $R>0$. Let $\eps>0$ be small. Set $x_j=j\eps$ for $j=1,...,p$. Then, as $\eps\to 0$,
\[\PP[(x_{j_1}+K_{j_1}\cap\U(0,R))\cap(x_{j_2}+K_{j_2}\cap\U(0,R))=\emptyset, 1\le j_1<j_2\le p]\approx \eps^{\hat{\xi}(\beta_1,...,\beta_p)}.\]
\end{proposition}

In the following theorem, we will consider the law of $K_1,...,K_p$ conditioned on ``non-intersection". Since the event of ``non-intersection" has zero probability, we need to explain the precise meaning: the conditioned law would be obtained through a limiting procedure: first consider the law of $K_1,...,K_p$ conditioned on
\[[(x_{j_1}+K_{j_1}\cap\U(0,R))\cap(x_{j_2}+K_{j_2}\cap\U(0,R))=\emptyset, 1\le j_1<j_2\le p],\]
and then let $R\to\infty$ and $\eps\to 0$.
\begin{theorem}\label{thm::halfplane_exponent}
Fix $\beta_1,...,\beta_p\ge 5/8$. Suppose $K_1,...,K_p$ are $p$ independent chordal restriction samples with exponents $\beta_1,...,\beta_p$ respectively. Then the ``fill-in" of the union of these $p$ sets conditioned on ``non-intersection" has the same law as chordal restriction sample of exponent $\tilde{\xi}({\beta_1,...,\beta_p})$.
\end{theorem}
For Proposition \ref{prop::disjoint_exponent} and Theorem \ref{thm::halfplane_exponent}, we only need to show the results for $p=2$ and other $p$ can be proved by induction. Proposition \ref{prop::disjoint_exponent} for $p=2$ is a direct consequence of the following lemma.
\begin{lemma}\label{lem::estimate_sle8over3_restriction}
Suppose $K$ is a right-sided restriction sample with exponent $\beta>0$. Let $\gamma$ be an independent chordal SLE$^R_{8/3}(\rho)$ process for some $\rho>-2$. Fix $t>0$ and let $\eps>0$ be small, we have
\[\PP[\gamma[0,t]\cap (K-\eps)=\emptyset]\approx\eps^{\frac{3}{16}\bar{\rho}(\rho+2)}\quad\text{as }\eps\to 0\]
where \[\bar{\rho}=\frac{2}{3}(\sqrt{24\beta+1}-1).\]

Note that, if $\beta_1=\beta$, $\beta_2=(3\rho^2+16\rho+20)/32$, we have
\[\frac{3}{16}\bar{\rho}(\rho+2)=\hat{\xi}(\beta_1,\beta_2).\]
\end{lemma}
\begin{proof}
Let $(g_t,t\ge 0)$ be the Loewner chain for $\gamma$ and $(O_t,W_t)$ be the solution to the SDE. Precisely,
\[\partial_tg_t(z)=\frac{2}{g_t(z)-W_t},\quad g_0(z)=z;\]
\[dW_t=\sqrt{\kappa}dB_t+\frac{\rho dt}{W_t-O_t},\quad dO_t=\frac{2dt}{O_t-W_t}, \quad W_0=O_0=0, \quad O_t\ge W_t.\]
Given $\gamma[0,t]$, since $K$ satisfies right-sided restriction property, we have that
\[\PP[\gamma[0,t]\cap (K-\eps)=\emptyset\cond \gamma[0,t]]=g_t'(-\eps)^{\beta}.\]
Define
\[M_t=g_t'(-\eps)^{\bar{\rho}(3\bar{\rho}+4)/32}(W_t-g_t(-\eps))^{3\bar{\rho}/8}(O_t-g_t(-\eps))^{3\bar{\rho}\rho/16}.\]
One can check that $M$ is a local martingale and $\beta=\bar{\rho}(3\bar{\rho}+4)/32$. Thus
\begin{align*}
\PP[(K-\eps)\cap\gamma[0,t]=\emptyset]&=\E\left[\PP[\gamma[0,t]\cap (K-\eps)=\emptyset\cond \gamma[0,t]]\right]\\
&=\E[g_t'(-\eps)^{\beta}]\\
&\approx\E[M_t]\\
&=M_0=\eps^{\frac{3}{16}\bar{\rho}(\rho+2)}.
\end{align*}
In this equation, the sign $\approx$ means that the ratio $\E[g_t'(-\eps)^{\beta}]/\E[M_t]$ corresponds to $\eps^{err}$ where the error term in the exponent $err$ goes to zero as $\eps$ goes to zero. In fact, we need more work to make this precise, we only show the key idea that how we get the correct exponent, and the details are left to interested readers.
\end{proof}
\begin{proof}[Proof of Theorem \ref{thm::halfplane_exponent}.]
Assume $p=2$. For any $A\in\LA_c$, we need to estimate the following probability for $\eps>0$ small:
\[\PP[K_1\cap A=\emptyset, K_2\cap A=\emptyset\cond (K_1\cap\U(0,R)+\eps)\cap(K_2\cap\U(0,R)+2\eps)=\emptyset].\]
For $i=1,2$, since $K_i$ satisfies chordal conformal restriction property, we know that the probability of $\{K_i\cap A=\emptyset\}$ is $\Phi_A'(0)^{\beta_i}$. Thus 
\begin{align*}
\PP&[K_1\cap A=\emptyset, K_2\cap A=\emptyset\cond (K_1\cap\U(0,R)+\eps)\cap(K_2\cap\U(0,R)+2\eps)=\emptyset]\\
&=\frac{\PP[K_1\cap A=\emptyset, K_2\cap A=\emptyset, (K_1\cap\U(0,R)+\eps)\cap(K_2\cap\U(0,R)+2\eps)=\emptyset]}{\PP[(K_1\cap\U(0,R)+\eps)\cap(K_2\cap\U(0,R)+2\eps)=\emptyset]}\\
&=\frac{\PP[(K_1\cap\U(0,R)+\eps)\cap(K_2\cap\U(0,R)+2\eps)=\emptyset\cond K_1\cap A=\emptyset, K_2\cap A=\emptyset]}{\PP[(K_1\cap\U(0,R)+\eps)\cap(K_2\cap\U(0,R)+2\eps)=\emptyset]}\\
&\quad \times\Phi_A'(0)^{\beta_1}\Phi_A'(0)^{\beta_2}.
\end{align*}
For $i=1,2$, conditioned on $[K_i\cap A=\emptyset]$, the conditional law of $\Phi_A(K_i)$ has the same law as $K_i$. Combining this with Proposition \ref{prop::disjoint_exponent}, we have
\begin{align*}
\lim_{R\to\infty}&
\frac{\PP[(K_1\cap\U(0,R)+\eps)\cap(K_2\cap\U(0,R)+2\eps)=\emptyset\cond K_1\cap A=\emptyset, K_2\cap A=\emptyset]}{\PP[(K_1\cap\U(0,R)+\eps)\cap(K_2\cap\U(0,R)+2\eps)=\emptyset]}\\
&=\lim_{R\to\infty}\frac{\PP[(K_1\cap\U(0,R)+\Phi_A(\eps))\cap(K_2\cap\U(0,R)+\Phi_A(2\eps))=\emptyset]}{\PP[(K_1\cap\U(0,R)+\eps)\cap(K_2\cap\U(0,R)+2\eps)=\emptyset]}\\
&\approx\frac{\Phi_A(\eps)^{\hat{\xi}(\beta_1,\beta_2)}}{\eps^{\hat{\xi}(\beta_1,\beta_2)}}.
\end{align*}
Therefore,
\begin{align*}
\lim_{R\to\infty,\eps\to 0}&\PP[K_1\cap A=\emptyset, K_2\cap A=\emptyset\cond (K_1\cap\U(0,R)+\eps)\cap(K_2\cap\U(0,R)+2\eps)=\emptyset]\\
&=\lim_{\eps\to 0}\Phi_A'(0)^{\beta_1}\Phi_A'(0)^{\beta_2}\frac{\Phi_A(\eps)^{\hat{\xi}(\beta_1,\beta_2)}}{\eps^{\hat{\xi}(\beta_1,\beta_2)}}\\
&=\Phi_A'(0)^{\beta_1+\beta_2+\hat{\xi}(\beta_1,\beta_2)}.
\end{align*}
This implies that, conditioned on ``non-intersection", the union $K_1\cup K_2$ satisfies chordal conformal restriction with exponent $\tilde{\xi}(\beta_1,\beta_2)$.
\end{proof}

%% file: tex/radial_sle.tex
\subsection{Radial Loewner chain}
\noindent\textbf{Capacity}
\medbreak
Consider a compact subset $K$ of $\overline{\U}$ such that $0\in\U\setminus K$ and $\U\setminus K$ is simply connected. Then there exists a unique conformal map $g_K$ from $\U\setminus K$ onto $\U$ normalized at the origin, i.e. $g_K(0)=0,g_K'(0)>0$. We call $a(K):=\log g_K'(0)$ the capacity of $K$ in $\U$ seen from the origin.
\begin{lemma}
The quantity $a$ is non-negative increasing function.
\end{lemma}
\begin{proof}
The quantity $a$ is non-negative: Denote $U=\U\setminus K$. Note that $\log g_K(z)/z$ is an analytic function on $U\setminus\{0\}$ and the origin is removable: we can define the function equals $\log g_K'(0)$ at the origin. Then $h(z)=\log|g_K(z)/z|$ is a harmonic function on $U$. Thus it attains its min on $\partial U$. For $z\in\partial U$, $h(z)\ge 0$. Therefore $h(z)\ge 0$ for all $z\in U$. In particular, $h(0)\ge 0$.

The quantity $a$ is increasing: Suppose $K\subset K'$. Define $g_1=g_K$ and let $g_2$ be the conformal map from $\U\setminus g_K(K'\setminus K)$ onto $\U$ normalized at the origin. Then $g_{K'}=g_2\circ g_1$. Thus
\[a(K')=\log g_2'(0)+\log g_1'(0)\ge \log g_1'(0)=a(K).\]
\end{proof}
\begin{remark}
If we denote by $d(0,K)$ the Euclidean distance from the origin to $K$, by Koebe $1/4$-Theorem, we have that
\[\frac{1}{4}e^{-a(K)}\le d(0,K)\le e^{-a(K)}.\]
\end{remark}
\medbreak
\noindent\textbf{Loewner chain}
\medbreak
Suppose $(W_t,t\ge 0)$ is a continuous real function with $W_0=0$. Define for $z\in\overline{\U}$, the function $g_t(z)$ as the solution to the ODE
\[\partial_t g_t(z)=g_t(z)\frac{e^{iW_t}+g_t(z)}{e^{iW_t}-g_t(z)},\quad g_0(z)=z.\]
The solution is well-defined as long as $e^{iW_t}-g_t(z)$ does not hit zero. Define
\[T(z)=\sup\{t>0: \min_{s\in[0,t]}|e^{iW_s}-g_s(z)|>0\}.\]
This is the largest time up to which $g_t(z)$ is well-defined. Set
\[K_t=\{z\in\overline{\U}: T(z)\le t\},\quad U_t=\U\setminus K_t.\]
We can check that the map $g_t$ is a conformal map from $U_t$ onto $\U$ normalized at the origin, and that, for each $t$, $g'_t(0)=e^t$. In other words, $a(K_t)=t$.
 The family $(K_t,t\ge 0)$ is called the \textbf{radial Loewner chain} driven by $(W_t,t\ge 0)$.
 
\subsection{Radial SLE}
\noindent\textbf{Definition}
\medbreak
Radial SLE$_{\kappa}$ for $\kappa\ge 0$ is the radial Loewner chain driven by $W_t=\sqrt{\kappa}B_t$ where $B$ is a 1-dimensional BM starting from $B_0=0$. This defines radial SLE$_{\kappa}$ in $\U$ from $1$ to the origin. For general simply connected domain $D$ with a boundary point $x$ and an interior point $z$, the radial SLE$_{\kappa}$ in $D$ from $x$ to $z$ is the image of radial SLE$_{\kappa}$ in $\U$ from $1$ to the origin under the conformal map from $\U$ to $D$ sending the pair $1,0$ to $x,z$. 
\begin{lemma}
Radial SLE satisfies domain Markov property: For any stopping time $T$, the process $(g_T(K_{t+T}\setminus K_T)e^{-iW_T}, t\ge 0)$ is independent of $(K_s,0\le s\le T)$ and has the same law as $K$.
\end{lemma}
\begin{proposition}
For $\kappa\in [0,4]$, radial SLE$_{\kappa}$ is almost surely a simple continuous curve. Moreover, almost surely, we have 
$\lim_{t\to\infty}\gamma(t)=0$.
\end{proposition}
\begin{figure}[ht!]
\begin{center}
\includegraphics[width=0.47\textwidth]{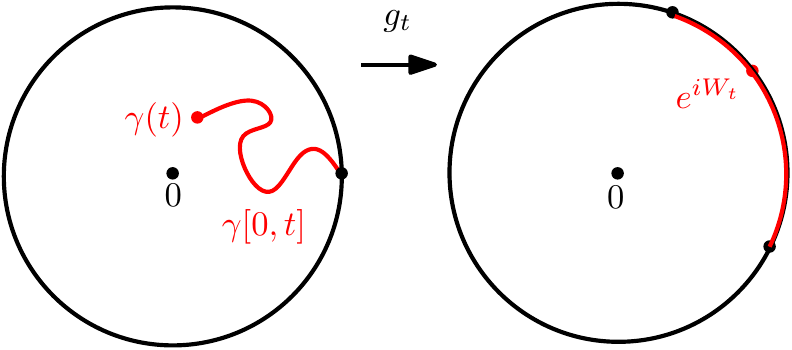}
\end{center}
\caption{\label{fig::sle_radial_explanation} The map $g_t$ is the conformal map from $\U\setminus \gamma[0,t]$ onto $\U$ normalized at the origin, and the tip of the curve $\gamma(t)$ is the preimage of $e^{iW_t}$ under $g_t$.}
\end{figure}
\medbreak
\noindent\textbf{Restriction property of radial SLE$_{8/3}$}
\medbreak
Recall Definition \ref{def::collection_radial}. Suppose $A\in\LA_r$, and $\Phi_A$ is the conformal map from $\U\setminus A$ onto $\U$ such that $\Phi_A(0)=0, \Phi_A(1)=1$.
Let $\gamma$ be a radial SLE$_{8/3}$, we will compute the probability $\PP[\gamma\cap A=\emptyset]$.
Similar as the chordal case, we need to study the image $\tilde{\gamma}=\Phi_A(\gamma)$. Define $T=\inf\{t:\gamma(t)\in A\}$, and For $t<T$, let $\tilde{\gamma}[0,t]:=\Phi_A(\gamma[0,t])$.
Note that $\Phi_A$ is the conformal map from $\U\setminus A$ onto $\U$ with $\Phi_A(0)=0,\Phi_A(1)=1$ and that $g_t$ is the conformal map from $\U\setminus\gamma[0,t]$ onto $\U$ normalized at the origin. Define $\tilde{g}_t$ to be the conformal map from $\U\setminus\tilde{\gamma}[0,t]$ onto $\U$ normalized at the origin and $h_t$ the conformal map from $\U\setminus g_t(A)$ onto $\U$ such that Equation (\ref{eqn::sle8over3_radial_exchange}) holds. See Figure \ref{fig::sle8over3_radial_exchange}.
\begin{equation}\label{eqn::sle8over3_radial_exchange}h_t\circ g_t=\tilde{g}_t\circ \Phi_A.\end{equation}
\begin{proposition}\label{prop::sle8over3_radial_martingale}
When $\kappa=8/3$, the process
\[M_t=|h_t'(0)|^{5/48}|h_t'(e^{iW_t})|^{5/8},\quad t<T\]
is a local martingale.
\end{proposition}
\begin{proof}
Define \[\phi_t(z)=-i\log h_t(e^{iz})\]
where $\log$ denotes the branch of the logarithm such that $-i\log h_t(e^{iW_t})=W_t$. Then
\[h_t(e^{iz})=e^{i\phi_t(z)},\quad h_t'(e^{iW_t})=\phi_t'(W_t).\]
Define
\[a(t)=a(K_t\cup A)=a(A)+a(\tilde{K}_t).\]
A similar time change argument as in the proof of Proposition \ref{prop::chordalsle8over3_localmart} shows that
\[\partial_t\tilde{g}_t(z)=\partial_t a\, \tilde{g}_t(z)\frac{e^{i\tilde{W}_t}+\tilde{g}_t(z)}{e^{i\tilde{W}_t}-\tilde{g}_t(z)}.\]
Plugging $h_t\circ g_t=\tilde{g}_t\circ\Phi_A$, we have
\begin{equation}\label{eqn::sle_radial_restriction_1}
\partial_t h_t(z)+h_t'(z)z\frac{e^{iW_t}+z}{e^{iW_t}-z}=\partial_ta\, h_t(z)\frac{e^{iW_t}+h_t(z)}{e^{iW_t}-h_t(z)}.\end{equation}
We can first find $\partial_t a$: multiply $e^{iW_t}-h_t(z)$ to both sides of Equation (\ref{eqn::sle_radial_restriction_1}) and then let $z\to e^{iW_t}$. We have
\[\partial_t a=h_t'(e^{iW_t})^2=\phi_t'(W_t)^2.\]
Denote
\[X_1=\phi_t'(W_t),\quad X_2=\phi_t''(W_t),\quad X_3=\phi_t'''(W_t).\]
Then Equation (\ref{eqn::sle_radial_restriction_1}) becomes
\[\partial_t h_t(z)=X_1^2h_t(z)\frac{e^{iW_t}+h_t(z)}{e^{iW_t}-h_t(z)}-h_t'(z)z\frac{e^{iW_t}+z}{e^{iW_t}-z}.\]
Plugin the relation $h_t(e^{iz})=e^{i\phi_t(z)}$, we have that
\begin{equation}\label{eqn::sle_radial_restriction_2}
\partial_t\phi_t(z)=X_1^2\cot(\frac{\phi_t(z)-W_t}{2})-\phi_t'(z)\cot(\frac{z-W_t}{2}).
\end{equation}
Differentiate Equation (\ref{eqn::sle_radial_restriction_2}) with respect to $z$, we have
\[\partial_t\phi_t'(z)=-\frac{1}{2}X_1^2\phi_t'(z)\csc^2(\frac{\phi_t(z)-W_t}{2})-\phi_t''(z)\cot(\frac{z-W_t}{2})+\frac{1}{2}\phi_t'(z)\csc^2(\frac{z-W_t}{2}).\]
Let $z\to W_t$,
\[\partial_t\phi_t'(W_t)=\frac{X_2^2}{2X_1}-\frac{4}{3}X_3+\frac{X_1-X_1^3}{6}.\]
Thus
\begin{equation}\label{eqn::sle_radial_restriction_w}dh_t'(e^{iW_t})=d\phi_t'(W_t)=\sqrt{\frac{8}{3}}X_2dB_t+(\frac{X_2^2}{2X_1}+\frac{X_1-X_1^3}{6})dt.\end{equation}
For the term $h_t'(0)$, we have that
\[h_t'(0)=\exp(a(g_t(A)))=\exp(a(t)-t)=\exp(a(A)+\int_0^t\phi_s'(W_s)^2ds-t),\]
thus
\begin{equation}\label{eqn::sle_radial_restriction_o}
dh_t'(0)=h_t'(0)(X_1^2-1)dt.
\end{equation}
Combining Equations (\ref{eqn::sle_radial_restriction_w}) and (\ref{eqn::sle_radial_restriction_o}), we have that
\[dM_t=\frac{5}{8}M_t\frac{X_2}{X_1}dW_t.\]

\end{proof}
\begin{figure}[ht!]
\begin{center}
\includegraphics[width=0.47\textwidth]{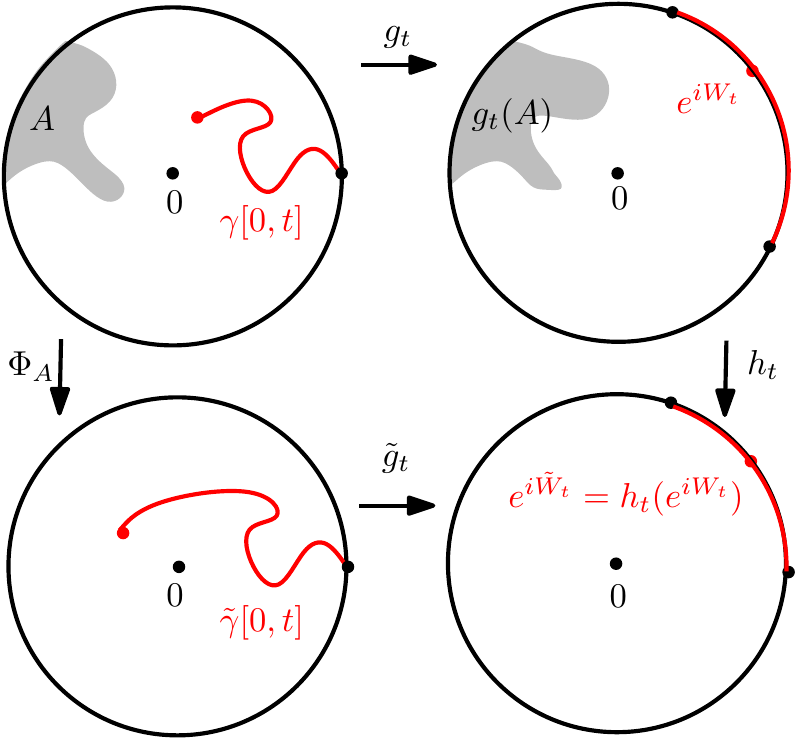}
\end{center}
\caption{\label{fig::sle8over3_radial_exchange} The map $\Phi_A$ is the conformal map from $\U\setminus A$ onto $\U$ with $\Phi_A(0)=0,\Phi_A(1)=1$. The map $g_t$ is the conformal map from $\U\setminus\gamma[0,t]$ onto $\U$ normalized at the origin. Define $\tilde{g}_t$ to be the conformal map from $\U\setminus\tilde{\gamma}[0,t]$ onto $\U$ normalized at the origin and $h_t$ the conformal map from $\U\setminus g_t(A)$ onto  $\U$ such that $h_t\circ g_t=\tilde{g}_t\circ\Phi_A$.}
\end{figure}
\begin{theorem}
Suppose $\gamma$ is a radial SLE$_{8/3}$ in $\U$ from $1$ to $0$. Then for any $A\in\LA_r$, we have
\[\PP[\gamma\cap A=\emptyset]=|\Phi_A'(0)|^{5/48}\Phi_A'(1)^{5/8}.\]
\end{theorem}
\begin{proof}
Suppose $M$ is the local martingale defined in Proposition \ref{prop::sle8over3_radial_martingale}. Note that
\[M_0=|\Phi_A'(0)|^{5/48}\Phi_A'(1)^{5/8}.\]
Define $T=\inf\{t: K_t\cap A\neq\emptyset\}.$
In fact, $|\Phi_A'(0)|\Phi_A'(1)^2\le 1$ for any $A\in\LA_r$, thus $M$ is a bounded martingale.

If $T=\infty$, we have 
\[\lim_{t\to\infty}h_t'(e^{iW_t})=1, \quad \lim_{t\to\infty}h_t'(0)=1, \quad \text{and }\lim_{t\to\infty}M_t=1.\]
If $T<\infty$, we have 
\[\lim_{t\to T}h_t'(e^{iW_t})=0, \quad \text{and }\lim_{t\to T}M_t=0.\]
Thus,
\[\PP[\gamma\cap A=\emptyset]=\PP[T=\infty]=E[M_T]=M_0.\]
\end{proof}
\subsection{Radial SLE$_{\kappa}(\rho)$ process}
Fix $\kappa>0$, $\rho>-2$. Radial SLE$_{\kappa}(\rho)$ process is the radial Loewner chain driven by $W$ which is the solution to the following SDE:
\begin{equation}\label{eqn::sle_radial_kapparho}
dW_t=\sqrt{\kappa}dB_t+\frac{\rho}{2}\cot(\frac{W_t-O_t}{2})dt,\quad dO_t=-\cot(\frac{W_t-O_t}{2})dt,
\end{equation}
with initial value $W_0=0, O_0=x\in (0,2\pi)$. 
When $\kappa>0,\rho>-2$, there exists a piecewise unique solution to the SDE (\ref{eqn::sle_radial_kapparho}). There exists almost surely a continuous curve $\gamma$ in $\overline{\U}$ from $1$ to $0$ so that $(K_t,t\ge 0)$ is generated by $\gamma$. When $\kappa\in [0,4]$ and $\rho\ge \kappa/2-2$, $\gamma$ is a simple curve and $K_t=\gamma[0,t]$. When $\kappa\in[0,4],\rho\in (-2,\kappa/2-2)$, $\gamma$ almost surely hits the boundary. The tip $\gamma(t)$ is the preimage of $e^{iW_t}$ under $g_t$, and $e^{ix}$ (when it is not swallowed by $K_t$) is the preimage of $e^{iO_t}$ under $g_t$. When $e^{ix}$ is swallowed by $K_t$, then the preimage of $e^{iO_t}$ under $g_t$ is the last point (before time $t$) on the curve that is on the boundary. See Figure \ref{fig::slekapparho_radial_explanation}.

\begin{figure}[ht!]
\begin{subfigure}[b]{0.47\textwidth}
\begin{center}
\includegraphics[width=\textwidth]{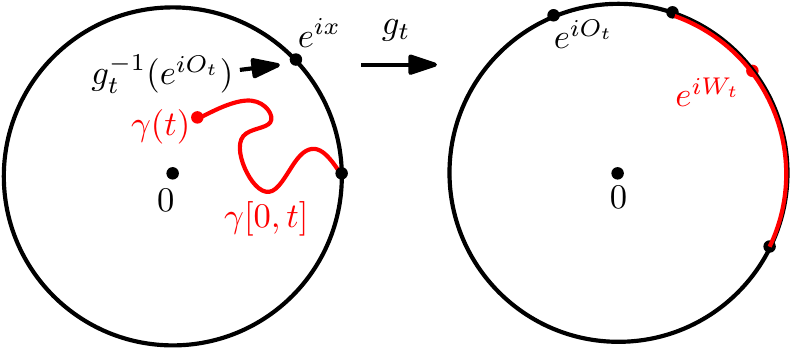}
\end{center}
\caption{When $\rho\ge \kappa/2-2$, the curve does not hit the boundary.}
\end{subfigure}
$\quad$
\begin{subfigure}[b]{0.47\textwidth}
\begin{center}\includegraphics[width=\textwidth]{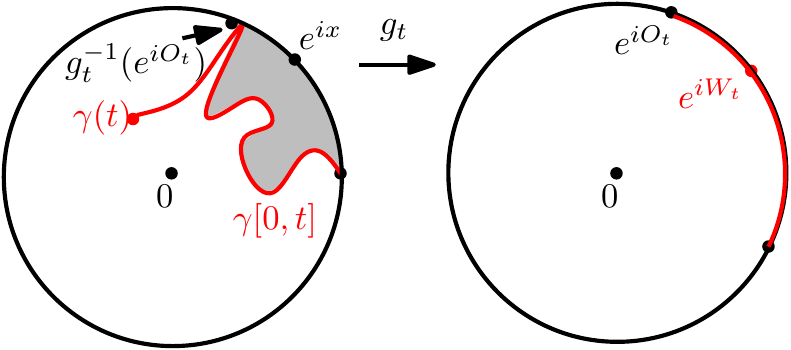}
\end{center}
\caption{When $\rho\in (-2,\kappa/2-2)$, the curve touches the boundary.}
\end{subfigure}
\caption{\label{fig::slekapparho_radial_explanation}Geometric meaning of $(O_t,W_t)$ in radial SLE$_{\kappa}(\rho)$ process. The preimage of $e^{iW_t}$ under $g_t$ is the tip of the curve, the preimage of $e^{iO_t}$ under $g_t$ is the last point on the curve that is on the boundary. }
\end{figure}

Let $x\to 0+$ (resp. $x\to 2\pi-$), the process has a limit, and we call this limit the radial SLE$^R_{\kappa}(\rho)$ (resp. SLE$^L_{\kappa}(\rho)$) in $\overline{\U}$ from $1$ to $0$.
Suppose $\gamma$ is an SLE$^L_{8/3}(\rho)$ process for some $\rho>-2$. For any $A\in\LA_r$, we want to analyze the image of $\gamma$ under $\Phi_A$. Define $T=\inf\{t:\gamma(t)\in A\}$. For $t<T$, note that $\Phi_A$ is the conformal map from $\U\setminus A$ onto $\U$ with $\Phi_A(0)=0,\Phi_A(1)=1$,  and that $g_t$ is the conformal map from $\U\setminus\gamma[0,t]$ onto $\U$ normalized at the origin. Define $\tilde{g}_t$ to be the conformal map from $\U\setminus\tilde{\gamma}[0,t]$ onto $\U$ normalized at the origin and $h_t$ the conformal map from $\U\setminus g_t(A)$ onto  $\U$ such that $h_t\circ g_t=\tilde{g}_t\circ\Phi_A$. See Figure \ref{fig::sle8over3rho_radial_exchange}. Denote
\[\theta_t=\frac{W_t-O_t}{2},\quad \vartheta_t=\frac{1}{2}\arg(h_t(e^{iW_t})/h_t(e^{iO_t})).\]
\begin{proposition}\label{prop::sle8over3rho_radial_mart}
Define \[M_t=|h_t'(0)|^{\alpha}\times|h_t'(e^{iW_t})|^{5/8}\times|h_t'(e^{iO_t})|^{\rho(3\rho+4)/32}\times\left(\frac{\sin\vartheta_t}{\sin\theta_t}\right)^{3\rho/8}\]
where
\[\alpha=\frac{5}{48}+\frac{3}{64}\rho(\rho+4).\] Then $M$ is a local martingale.
Note that, if we set
\[\beta=\frac{5}{8}+\frac{1}{32}\rho(3\rho+4)+\frac{3}{8}\rho,\] we have $\alpha=\xi(\beta)$.
\end{proposition}
\begin{figure}[ht!]
\begin{center}
\includegraphics[width=0.47\textwidth]{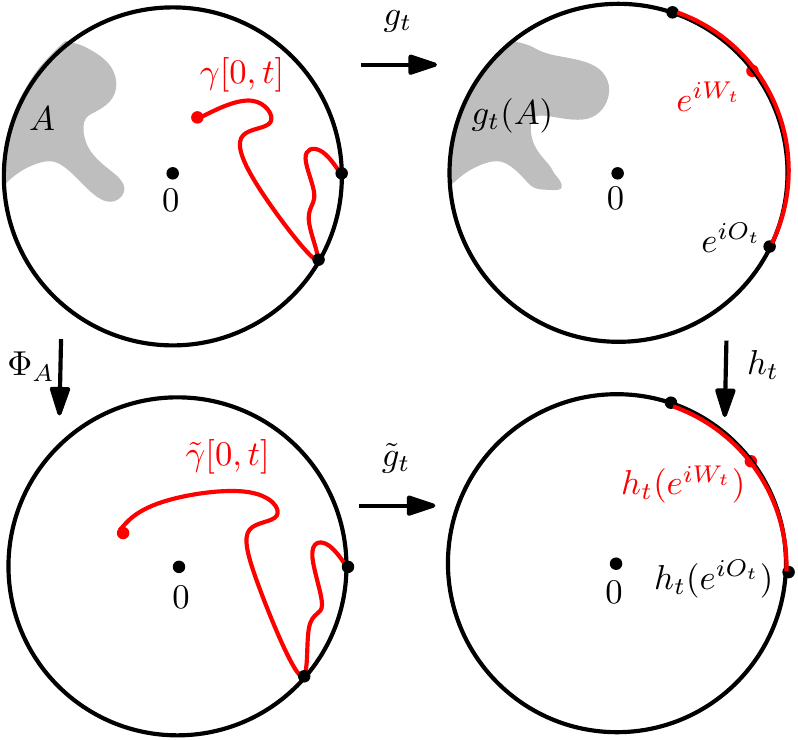}
\end{center}
\caption{\label{fig::sle8over3rho_radial_exchange} The map $\Phi_A$ is the conformal map from $\U\setminus A$ onto $\U$ with $\Phi_A(0)=0,\Phi_A(1)=1$. The map $g_t$ is the conformal map from $\U\setminus\gamma[0,t]$ onto $\U$ normalized at the origin. Define $\tilde{g}_t$ to be the conformal map from $\U\setminus\tilde{\gamma}[0,t]$ onto $\U$ normalized at the origin and $h_t$ the conformal map from $\U\setminus g_t(A)$ onto  $\U$ such that $h_t\circ g_t=\tilde{g}_t\circ\Phi_A$.}
\end{figure}

\begin{proof}
Define $\phi_t(z)=-i\log h_t(e^{iz})$ where $\log$ denotes the branch of the logarithm such that $-i\log h_t(e^{iW_t})=W_t$. Then
\[|h'_t(e^{iW_t})|=\phi'_t(W_t),\quad |h'_t(e^{iV_t})|=\phi'_t(V_t),\quad \vartheta_t=(\phi_t(W_t)-\phi_t(V_t))/2.\]
To simplify the notations, we set $X_1=\phi'_t(W_t), X_2=\phi''_t(W_t), Y_1=\phi'_t(V_t)$. By It\^o's Formula, we have that
\begin{equation*}
\begin{split}
d\phi_t(W_t)&=\sqrt{8/3}X_1dB_t+\left(-\frac{5}{3}X_2+\frac{\rho}{2}X_1\cot\theta_t\right)dt,\\
d\phi_t(V_t)&=-X_1^2\cot\vartheta_tdt,\\
d\phi'_t(W_t)&=\sqrt{8/3}X_2dB_t+\left(\frac{\rho}{2}X_2\cot\theta_t+\frac{X_2^2}{2X_1}+\frac{X_1-X_1^3}{6}\right)dt, \\
d\phi'_t(V_t)&=\left(-\frac{1}{2}X_1^2Y_1\frac{1}{\sin^2\vartheta_t}+\frac{1}{2}Y_1\frac{1}{\sin^2\theta_t}\right)dt,\\
d\theta_t&=\frac{\sqrt{8/3}}{2}dB_t + \frac{\rho+2}{4}\cot\theta_t dt,\\
d\vartheta_t&=\frac{\sqrt{8/3}}{2}X_1dB_t+\left(-\frac{5}{6}X_2+\frac{1}{2}X_1^2\cot\vartheta_t+\frac{\rho}{4}X_1\cot\theta_t\right)dt.
\end{split}\end{equation*}
Combining these identities,  we see that $M$ is a local martingale.
\end{proof}

\subsection{Relation between radial SLE and chordal SLE}\label{subsec::sle_radial_chordal}
Roughly speaking, chordal SLE is the limit of radial SLE when we let the interior target point go towards a boundary target point. Precisely, for $z\in\HH$, suppose $\varphi^z$ is the Mobius transformation from $\U$ onto $\HH$ that sends $0$ to $z$ and $1$ to $0$. We define radial SLE in $\HH$ from $0$ to $z$ as the image of radial SLE in $\U$ from 1 to 0 under $\varphi^z$. Then, as $y\to\infty$, radial SLE$_{\kappa}$ in $\HH$ from $0$ to $iy$ will converge to chordal SLE$_{\kappa}$ (under an appropriate topology).

\begin{proof}
Fix $R>0$, suppose $y>0$ large. Let $\gamma^{iy}$ be a radial SLE$_{\kappa}$ in $\HH$ from $0$ to $iy$ and let $\gamma$ be a chordal SLE$_{\kappa}$ in $\HH$ from $0$ to $\infty$. Let $\tau_R$ be the first time that the curve exits $\U(0,R)$. Set $\rho=6-\kappa$, and define
\[M_t(iy)=|g_t'(iy)|^{\rho(\rho+8-2\kappa)/(8\kappa)}(\Im g_t(iy))^{\rho^2/(8\kappa)}|g_t(iy)-W_t|^{\rho/\kappa}.\]
One can check that $M$ is a local martingale under the law of $\gamma$ (see \cite[Theorem 6]{SchrammWilsonSLECoordinatechanges}). Moreover, the measure weighted by $M(iy)/M_0(iy)$ is the same as the law of $\gamma^{iy}$ (after time-change). In particular, the Radon-Nikodym between the law of $\gamma^{iy}[0,\tau_R]$ and the law of $\gamma[0,\tau_R]$ is given by
\begin{eqnarray*}
\lefteqn{M_{\tau_R}(iy)/M_0(iy)}\\
&=&|g_{\tau_R}'(iy)|^{\rho(\rho+8-2\kappa)/(8\kappa)}\left(\frac{\Im g_{\tau_R}(iy)}{y}\right)^{\rho^2/(8\kappa)}\left(\frac{|g_{\tau_R}(iy)-W_{\tau_R}|}{y}\right)^{\rho/\kappa}
\end{eqnarray*}
which converges to 1 as $y\to\infty$.
\end{proof}

%% file: tex/radial_restriction.tex
\subsection{Setup for radial restriction sample}
Let $\Omega$ be the collection of compact subset $K$ of $\overline{\U}$ such that
\[K\cap\partial\U=\{1\}, 0\in K, K \text{ is connected and } \U\setminus K \text{ is connected}.\]
Recall $\LA_r$ in Definition \ref{def::collection_radial}. Endow $\Omega$ with the $\sigma$-field generated by the events $[K\in\Omega: K\cap A=\emptyset]$ where $A\in\LA_r$. Clearly, a probability measure $\PP$ on $\Omega$ is characterized by the values of $\PP[K\cap A=\emptyset]$ for $A\in\LA_r$.
\begin{definition} A probability measure $\PP$ on $\Omega$ is said to satisfy radial restriction property if the following is true.
For any $A\in\LA_r$, $\Phi_A(K)$ conditioned on $[K\cap A=\emptyset]$ has the same law as $K$.
\end{definition}
\begin{theorem}\label{thm::radial_restriction}
Radial restriction measure have the following description.
\begin{enumerate}
\item [(1)] (Characterization) A radial restriction measure is characterized by a pair of real numbers $(\alpha,\beta)$ such that, for every $A\in\LA_r$,
\begin{equation}\label{eqn::radial_restriction_characterization}\PP[K\cap A=\emptyset]=|\Phi_A'(0)|^{\alpha}\Phi_A'(1)^{\beta}.\end{equation}
We denote the corresponding radial restriction measure by $\QQ(\alpha,\beta)$.
\item [(2)] (Existence) The measure $\QQ(\alpha,\beta)$ exists if and only if
\[\beta\ge 5/8, \alpha\le \xi(\beta)=\frac{1}{48}((\sqrt{24\beta+1}-1)^2-4).\]
\end{enumerate}
\end{theorem}
\begin{remark} We already know the existence of $\QQ(5/48,5/8)$ when $K$ is radial SLE$_{8/3}$. Recall Theorem \ref{thm::ppp_bl_restriction}, if we take an independent Poisson point process with intensity $\alpha\mu^{loop}_{\U,0}$, the ``fill-in" of the union of the Poisson point process and radial SLE$_{8/3}$ would give $\QQ(5/48-\alpha,5/8)$.
\end{remark}
\begin{remark}
In Equation (\ref{eqn::radial_restriction_characterization}), we have that $|\Phi_A'(0)|\ge 1$ and $\Phi_A'(1)\le 1$. Since $\beta$ is positive, we have $\Phi_A'(1)^{\beta}\le 1$. But $\alpha$ can be negative or positive, so that $|\Phi_A'(0)|^{\alpha}$ can be greater than 1. The product $|\Phi_A'(0)|^{\alpha}\Phi_A'(1)^{\beta}$ is always less than 1 which is guaranteed by the condition that $\alpha\le\xi(\beta)$. (In fact, we always have $|\Phi_A'(0)|\Phi_A'(1)^2\le 1$.)
\end{remark}
\begin{remark}
While the class of chordal restriction measures is characterized by one single parameter $\beta\ge 5/8$, the class of radial restriction measures involves the additional parameter $\alpha$. This is due to the fact that the radial restriction property is in a sense weaker than the chordal one: the chordal restriction samples in $\HH$ are scale-invariant, while the radial ones are not.
\end{remark}

\subsection{Proof of Theorem \ref{thm::radial_restriction}. Characterization.}
Suppose that $K$ satisfies Equation (\ref{eqn::radial_restriction_characterization}) for any $A\in\LA_r$, then $K$ satisfies radial restriction property. Thus, we only need to show that there exist only two-degree of freedom for the radial restriction measures. 

It is easier to carry out the calculation in the upper half-plane $\HH$ instead of $\U$. Suppose that $K$ satisfies radial restriction property in $\HH$ with interior point $i$ and the boundary point $0$. In other words, $K$ is the image of radial restriction sample in $\U$ under the conformal map $\varphi(z)=i(1-z)/(1+z)$. The proof consists of six steps. 
\medbreak
\noindent\textbf{Step 1.} For any $x\in \R\setminus\{0\}$, the probability $\PP[K\cap \U(x,\eps)\neq\emptyset]$ decays like $\eps^2$ as $\eps$ goes to zero. And the limit
\[\lim_{\eps\to 0}\PP[K\cap \U(x,\eps)\neq\emptyset]/\eps^2\]
exists which we denote by $\lambda(x)$. Furthermore $\lambda(x)\in (0,\infty)$.
\medbreak
\noindent\textbf{Step 2.} The function $\lambda$ is continuous and differentiable for $x\in (-\infty,0)\cup(0,\infty)$.
We omit the proof for Steps 1 and 2, and the interested readers could consult \cite{WuConformalRestrictionRadial}.
\medbreak
\noindent\textbf{Step 3.} Fix $x,y\in \R\setminus\{0\}$. We estimate the probability of
\[\PP[K\cap \U(x,\eps)\neq\emptyset, K\cap \U(y,\delta)\neq\emptyset].\]
Clearly, it decays like $\eps^2\delta^2$ as $\eps,\delta$ go to zero. Define $f_{x,\eps}$ to be the conformal map from $\HH\setminus \U(x,\eps)$ onto $\HH$ that fixes 0 and $i$. In fact,
we can write out the exact expression of $f_{x,\eps}$: suppose $0<\eps<|x|$. Then
\[g_{x,\eps}(z):=z+\frac{\eps^2}{z-x}\]
is a conformal map from $\HH\setminus \U(x,\eps)$ onto $\HH$. Define
\[f_{x,\eps}(z)= b\frac{g_{x,\eps}(z)-c}{b^2+(c-a)(g_{x,\eps}(z)-a)}\]
where $a=\Re(g_{x,\eps}(i)), b=\Im(g_{x,\eps}(i)), c=g_{x,\eps}(0)$. Then $f_{x,\eps}$ is the conformal map from $\HH\setminus \U(x,\eps)$ onto $\HH$ that preserves 0 and $i$.
\begin{lemma}\label{lem::two_point_estimate}
\begin{eqnarray*}
\lefteqn{\lim_{\eps\to 0}\lim_{\delta\to 0}\frac{1}{\eps^2\delta^2}\PP[K\cap \U(x,\eps)\neq\emptyset, K\cap \U(y,\delta)\neq\emptyset]}\\
&=&\lambda(x)\lambda(y)-\lambda'(y)F(x,y)-2\lambda(y)G(x,y)
\end{eqnarray*}
where
\[F(x,y)=\lim_{\eps\to 0}\frac{1}{\eps^2}(f_{x,\eps}(y)-y)=\frac{1+x^2+y^2+xy}{x(1+x^2)}+\frac{1}{y-x},\]
\[G(x,y)=\lim_{\eps\to 0}\frac{1}{\eps^2}(f'_{x,\eps}(y)-1)=\frac{x+2y}{x(1+x^2)}-\frac{1}{(y-x)^2}.\]
\end{lemma}
\begin{proof}
By radial restriction property, we have that
\begin{eqnarray*}
\lefteqn{\PP[K\cap \U(x,\eps)=\emptyset, K\cap \U(y,\delta)\neq\emptyset]}\\
&=&\PP[K\cap \U(x,\eps)=\emptyset]\times \PP[K\cap \U(y,\delta)\neq\emptyset\cond K\cap \U(x,\eps)=\emptyset]\\
&=&\PP[K\cap \U(x,\eps)=\emptyset]\times \PP[K\cap f_{x,\eps}(\U(y,\delta))\neq\emptyset]
\end{eqnarray*}
Thus,
\begin{eqnarray*}
\lefteqn{\lim_{\eps\to 0}\lim_{\delta\to 0}\frac{1}{\eps^2\delta^2}\PP[K\cap \U(x,\eps)\neq\emptyset, K\cap \U(y,\delta)\neq\emptyset]}\\
&=&\lim_{\eps\to 0}\lim_{\delta\to 0}\frac{1}{\eps^2\delta^2}\\
&&\quad \times\left(\PP[K\cap \U(y,\delta)\neq\emptyset]-\PP[K\cap \U(x,\eps)=\emptyset]\times \PP[K\cap f_{x,\eps}(\U(y,\delta))\neq\emptyset]\right)\\
&=&\lim_{\eps\to 0}\frac{1}{\eps^2}\left(\lambda(y)-\PP[K\cap \U(x,\eps)=\emptyset]\lambda(f_{x,\eps}(y))|f'_{x,\eps}(y)|^2\right)\\
&=&\lim_{\eps\to 0}\frac{1}{\eps^2}\left(\PP[K\cap \U(x,\eps)\neq\emptyset]\lambda(f_{x,\eps}(y))|f'_{x,\eps}(y)|^2+\lambda(y)-\lambda(f_{x,\eps}(y))|f'_{x,\eps}(y)|^2\right)\\
&=&\lambda(x)\lambda(y)-\lambda'(y)F(x,y)-2\lambda(y)G(x,y).
\end{eqnarray*}
\end{proof}
\medbreak
\noindent\textbf{Step 4.} We are allowed to exchange the order in taking the limits in Lemma \ref{lem::two_point_estimate}. In other words, we have that
\begin{equation}\label{eqn::commutation_relation}
\lambda'(y)F(x,y)+2\lambda(y)G(x,y)=\lambda'(x)F(y,x)+2\lambda(x)G(y,x).
\end{equation}
We call this equation the \textbf{Commutation Relation}.
\medbreak
\noindent\textbf{Step 5.} Solve the function $\lambda$ through Commutation Relation (\ref{eqn::commutation_relation}).
\begin{lemma}\label{lem::lambda_expression} There exists two constants $c_0\ge 0,c_2\ge 0$ such that
\[\lambda(x)=\frac{c_0+c_2x^2}{x^2(1+x^2)^2}.\]
\end{lemma}
\begin{proof}
In Commutation Relation (\ref{eqn::commutation_relation}), let $y\to x$, we obtain a differential equation for the function $\lambda$. To write it in a better way, define
\[P(x)=x^2(1+x^2)^2\lambda(x),\] then the differential equation becomes
\[P'''(x)=0.\] 
Moreover, we know that $\lambda$ is an even function.
Thus, there exist three constants $c_0,c_1,c_2$ such that
\[\lambda(x)=\frac{c_0+c_1x+c_2x^2}{x^2(1+x^2)^2},\quad\text{for }x>0;\]
\[\lambda(x)=\frac{c_0-c_1x+c_2x^2}{x^2(1+x^2)^2},\quad\text{for }x<0.\]
We plugin these identities in Commutation Relation (\ref{eqn::commutation_relation}) and take $x>0>y$, then we get $c_1=0$. Since $\lambda$ is positive, we have $c_0\ge 0, c_2\ge 0$.
\end{proof}
\medbreak
\noindent\textbf{Step 6.} Find the relation between $(\alpha,\beta)$ and $(c_0,c_2)$.
Since there are only two-degree of freedom, when $K$ satisfies radial restriction property, we must have that Equation (\ref{eqn::radial_restriction_characterization}) holds for some $\alpha,\beta$. Note that
\[\PP[K\cap \U(x,\eps)\neq\emptyset]\sim\lambda(x)\eps^2.\]
Compare it with
\[1-|f_{x,\eps}'(i)|^{\alpha}f_{x,\eps}'(0)^{\beta},\] we have that
\[\alpha=(c_0-c_2)/4,\quad \beta=c_0/2.\]

\subsection{Several basic observations of radial restriction property}
Recall a result for Brownian loop: Theorem \ref{thm::ppp_bl_restriction}. Let $(l_j,j\in J)$ be a Poisson point process with intensity $c\mu^{loop}_{\U,0}$ for some $c>0$. Set $\Sigma=\cup_{j}l_j$. Then we have that, for any $A\in\LA_r$,
\[\PP[\Sigma\cap A=\emptyset]=|\Phi_A'(0)|^{-c}.\]
Suppose $K_0$ is a radial restriction sample whose law is $\QQ(\alpha_0,\beta_0)$. Take $K$ as the ``fill-in" of the union of $\Sigma$ and $K_0$, then clearly, $K$ has the law of $\QQ(\alpha_0-c,\beta_0)$. Thus we derived the following lemma.
\begin{lemma}\label{lem::radial_restriction_smallalpha}
If the radial restriction measure exists for some $(\alpha_0,\beta_0)$, then $\QQ(\alpha,\beta_0)$ exists for all $\alpha<\alpha_0$. Furthermore, almost surely for $\QQ(\alpha,\beta_0)$, the origin is not on the boundary of $K$.
\end{lemma}
In the next subsection, we will construct $\QQ(\xi(\beta),\beta)$ for $\beta\ge 5/8$ and point out that if $K$ has the law of $\QQ(\xi(\beta),\beta)$, almost surely the origin is on the boundary of $K$. Thus, combine with Lemma \ref{lem::radial_restriction_smallalpha}, we could show that, when $\beta\ge 5/8$, $\QQ(\alpha,\beta)$ exists if and only if $\alpha\le \xi(\beta)$. Note that, radial SLE$_{8/3}$ has the same law as $\QQ(5/48,5/8)$ where $5/48=\xi(5/8)$.
\medbreak
Another basic observation is that $\QQ(\alpha,\beta)$ does not exist when $\beta<5/8$. Suppose $K^0$ is a radial restriction sample with law $\QQ(\alpha,\beta)$. For any interior point $z\in\HH$, we define $K^z$ as the image of $K^0$ under the Mobius transformation from $\U$ onto $\HH$ such that sends 1 to $0$ and $0$ to $z$. Similar as the relation between radial SLE and chordal SLE in Subsection \ref{subsec::sle_radial_chordal},  if we let $z\to\infty$, $K^z$ converges weakly toward some probability measure, and the limit measure satisfies chordal restriction property with exponent $\beta$, thus $\beta\ge 5/8$.

\subsection{Construction of radial restriction measure $\QQ(\xi(\beta),\beta)$ for $\beta>5/8$}
The construction of $\QQ(\xi(\beta),\beta)$ is very similar to the construction of $\PP(\beta)$.
\begin{proposition}
Fix $\beta>5/8$ and let
\[\rho=\rho(\beta)=\frac{1}{3}(-8+2\sqrt{24\beta+1}).\]
Let $\gamma^R$ be a radial SLE$^L_{8/3}(\rho)$ in $\bar{\U}$ from 1 to 0. Given $\gamma^R$, let $\gamma^L$ be an independent chordal SLE$^R_{8/3}(\rho-2)$ in $\overline{\U\setminus\gamma^R}$ from $1^-$ to 0. Let $K$ be the closure of the union of the domains between $\gamma^L$ and $\gamma^R$. See Figure \ref{fig::construction_maximal_sample}. Then the law of $K$ is $\QQ(\xi(\beta),\beta)$. In particular, the origin is almost surely on the boundary of $K$.
\end{proposition}

\begin{figure}[ht!]
\begin{center}
\includegraphics[width=0.23\textwidth]{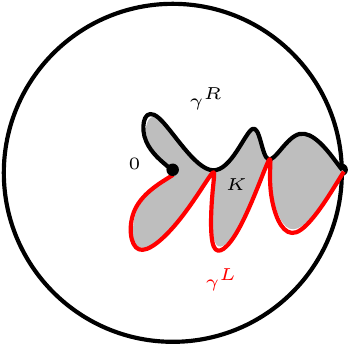}
\end{center}
\caption{\label{fig::construction_maximal_sample} The curve $\gamma^R$ is a radial SLE$^L_{8/3}(\rho)$ in $\U$ from 1 to 0. Conditioned on $\gamma^R$, the curve $\gamma^L$ is a chordal SLE$^R_{8/3}(\rho-2)$ in $\overline{\U\setminus\gamma^R}$ from $1^-$ to 0. The set $K$ is the closure of the union of domains between the two curves.}
\end{figure}
\begin{figure}[ht!]
\begin{center}
\includegraphics[width=0.47\textwidth]{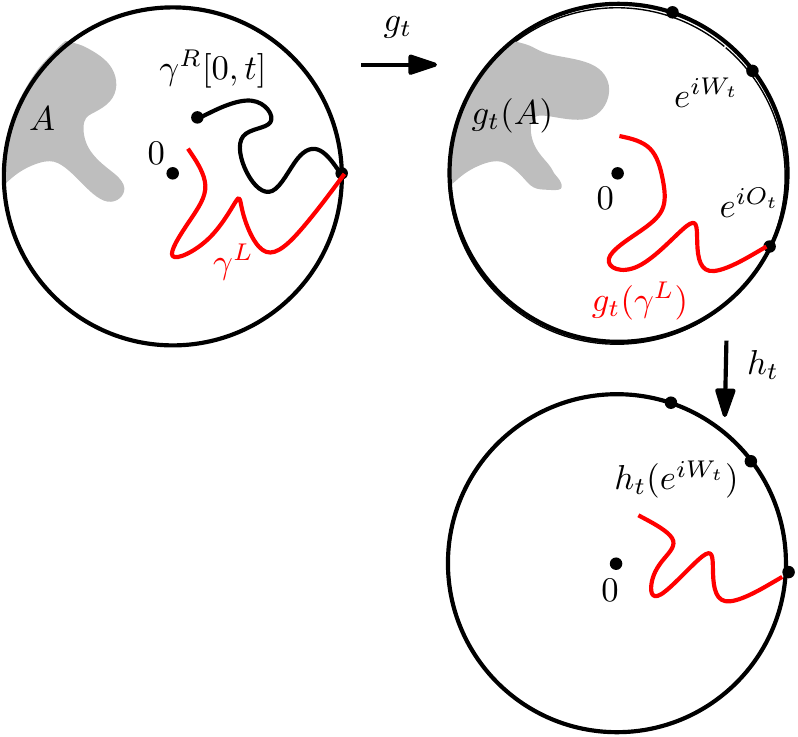}
\end{center}
\caption{\label{fig::radial_restriction_construction} The map $g_t$ is the conformal map from $\U\setminus\gamma^R[0,t]$ onto $\U$ normalized at the origin. The map $h_t$ is the conformal map from $\U\setminus g_t(A)$ onto $\U$ normalized at the origin.}
\end{figure}
\begin{proof}
Suppose $(g_t,t\ge 0)$ is the Loewner chain for $\gamma^R$. For any $A\in\LA_r$, define $T=\inf\{t: \gamma^R(t)\in A\}$. For $t<T$, let $h_t$ be the conformal map from $\U\setminus g_t(A)$ onto $U$ normalized at the origin.
From Proposition \ref{prop::sle8over3rho_radial_mart}, we know that
\[M_t=|h_t'(0)|^{\alpha}\times|h_t'(e^{iW_t})|^{5/8}\times|h_t'(e^{iO_t})|^{\rho(3\rho+4)/32}\times\left(\frac{\sin\vartheta_t}{\sin\theta_t}\right)^{3\rho/8}\]
is a local martingale, and
\[M_0=|\Phi_A'(0)|^{\xi(\beta)}\Phi_A'(1)^{\beta}.\]
See Figure \ref{fig::radial_restriction_construction}. If $T<\infty$,
\[\lim_{t\to T}h_t'(e^{iW_t})=0, \quad \text{and }\lim_{t\to T}M_t=0.\]
If $T=\infty$, as $t\to\infty$
\[|h_t'(0)|\to 1,\quad |h_t'(e^{iW_t})|\to 1,\quad \frac{\sin\theta_t}{\sin\vartheta_t}\to 1,\]
\[
|h_t'(e^{iO_t})|^{\rho(3\rho+4)/32}\to \PP[\gamma^L\cap A=\emptyset\cond \gamma^R].\]
Thus,
\[\PP[K\cap A=\emptyset]=\E[1_{T=\infty}\E[1_{K\cap A=\emptyset}]\cond \gamma^R]=\E[M_T]=M_0.\]

\end{proof}

\subsection{Whole-plane intersection exponents $\xi(\beta_1,...,\beta_p)$}\label{subsec::restriction_exponents_radial}
Recall that $\xi(\beta_1,...,\beta_p)$ and $\tilde{\xi}(\beta_1,...,\beta_p)$ are defined in Equations (\ref{eqn::wholeplane_exponent}, \ref{eqn::halfplane_exponent}) and $\hat{\xi}(\beta_1,...,\beta_p)$ is defined in Equation (\ref{eqn::hatxi}). 
For $x\in\R$ and a subset $K\subset\overline{\U}$, denote
\[e^{ix}K=\{e^{ix}z: z\in K\}.\]
For $r\in (0,1)$, denote the annulus by
\[\A_r=\U\setminus\overline{\U}(0,r).\]
\begin{proposition} \label{prop::disjoint_exponent_radial}
Fix $\beta_1,...,\beta_p\ge 5/8$. Suppose $K_1,...,K_p$ are $p$ independent radial restriction samples whose laws are $\QQ(\xi(\beta_1),\beta_1)$,..., $\QQ(\xi(\beta_p),\beta_p)$ respectively. Let $r>0$, $\eps>0$ small. Set $x_j=j\eps$ for $j=1,...,p$. Then, as $\eps,r\to 0$, we have
\begin{eqnarray*}
\lefteqn{\PP[(e^{ix_{j_1}}K_{j_1}\cap\A_r)\cap (e^{ix_{j_2}}K_{j_2}\cap\A_r)=\emptyset, 1\le j_1<j_2\le p]}\\
&\approx& \eps^{\hat{\xi}(\beta_1,...,\beta_p)}r^{\xi(\beta_1,...,\beta_p)-\xi(\beta_1)-\cdots-\xi(\beta_p)}.
\end{eqnarray*}
\end{proposition}
In the following theorem, we will consider the law of $K_1,...,K_p$ conditioned on ``non-intersection". Since the event of ``non-intersection" has zero probability, we need to explain the precise meaning: the conditioned law would be obtained through a limiting procedure: first consider the law of $K_1,...,K_p$ conditioned on
\[[(e^{ix_{j_1}}K_{j_1}\cap\A_r)\cap (e^{ix_{j_2}}K_{j_2}\cap\A_r)=\emptyset, 1\le j_1<j_2\le p]\]
and then let $r\to 0$ and $\eps\to 0$.
\begin{theorem}\label{thm::wholeplane_exponent}
Fix $\beta_1,...,\beta_p\ge 5/8$. Suppose $K_1,...,K_p$ are $p$ independent radial restriction samples whose laws are $\QQ(\xi(\beta_1),\beta_1)$,..., $\QQ(\xi(\beta_p),\beta_p)$ respectively. Then the ``fill-in" of the union of these $p$ sets conditioned on ``non-intersection" has the same law as radial restriction sample with law \[\QQ(\xi(\beta_1,...,\beta_p),\tilde{\xi}(\beta_1,...,\beta_p)).\]
\end{theorem}
We only need to show the results for $p=2$ and other $p$ can be proved by induction. When $p=2$, Proposition \ref{prop::disjoint_exponent_radial} is a direct consequence of the following lemma.
\begin{lemma}\label{lem::estimate_sle8over3_restriction_radial}
Let $K$ be a radial restriction sample with exponents $(\alpha,\beta)$. Let $r>0$, $\eps>0$ be small. Suppose $\gamma$ is an independent radial SLE$^L_{8/3}(\rho)$ process. Then we have
\[\PP[\gamma[0,t]\cap (e^{i\eps}K)=\emptyset]\approx\eps^{\frac{3}{16}\bar{\rho}(\rho+2)}r^{\bar{q}-q-\alpha}\quad\text{as }\eps,r\to 0\]
where $r=e^{-t}$, and
\[\bar{\rho}=\frac{2}{3}(\sqrt{24\beta+1}-1),\]
\[q=\frac{3}{64}\rho(\rho+4), \quad \bar{q}=\frac{3}{64}(\bar{\rho}+\rho)(\bar{\rho}+\rho+4).\]
Note that, if $\beta_1=\beta$, $\beta_2=(3\rho^2+16\rho+20)/32,\alpha=\xi(\beta_1)$, then
\[\frac{3}{16}\bar{\rho}(\rho+2)=\hat{\xi}(\beta_1,\beta_2),\quad \bar{q}-q-\alpha=\xi(\beta_1,\beta_2)-\xi(\beta_1)-\xi(\beta_2).\]
\end{lemma}
\begin{proof}
Let $(g_t,t\ge 0)$ be the Loewner chain for $\gamma$ and $(O_t,W_t)$ be the solution to the SDE. Precisely,
\[\partial_t g_t(z)=g_t(z)\frac{e^{iW_t}+g_t(z)}{e^{iW_t}-g_t(z)},\quad g_0(z)=z;\]
\[dW_t=\sqrt{\kappa}dB_t+\frac{\rho}{2}\cot(\frac{W_t-O_t}{2})dt,\quad dO_t=-\cot(\frac{W_t-O_t}{2})dt,\quad W_0=0, O_0=2\pi-.\]
Given $\gamma[0,t]$, since $K$ satisfies radial restriction property, we have that
\[\PP[\gamma[0,t]\cap(e^{i\eps}K)=\emptyset\cond \gamma[0,t]]=g_t'(e^{i\eps})^{\beta}e^{t\alpha}.\]
Define
\[M_t=e^{t(\bar{q}-q)}g_t'(e^{i\eps})^{\beta}|g_t(e^{i\eps})-e^{iW_t}|^{3\bar{\rho}/8}|g_t(e^{i\eps})-e^{iO_t}|^{3\rho\bar{\rho}/16}.\]
One can check that $M$ is a local martingale. Thus we have
\begin{eqnarray*}
\lefteqn{\PP[\gamma[0,t]\cap(e^{i\eps}K)=\emptyset]}\\
&=&\E[g_t'(e^{i\eps})^{\beta}e^{t\alpha}]\\
&=&e^{t(\alpha-\bar{q}+q)}\E[e^{t(\bar{q}-q)}g_t'(e^{i\eps})^{\beta}]\\
&\approx& r^{\bar{q}-q-\alpha}\E[M_t]=r^{\bar{q}-q-\alpha}M_0.\end{eqnarray*}
\end{proof}
\begin{proof}[Proof of Theorem \ref{thm::wholeplane_exponent}] Assume $p=2$. For any $A\in\LA_r$, we need to estimate the following probability
\[\PP[K_1\cap A=\emptyset, K_2\cap A=\emptyset\cond (e^{i\eps}K_1\cap\A_r)\cap (e^{i2\eps}K_2\cap\A_r)=\emptyset].\]
The idea is similar to the proof of Theorem \ref{thm::halfplane_exponent}. 
Since $K_i$ satisfies radial restriction property, conditioned on $[K_i\cap A=\emptyset]$, the conditional law of $\Phi_A(K_i)$ has the same law as $K_i$ for $i=1,2$. Thus
\begin{eqnarray*}
\lefteqn{\lim_{\eps\to 0,r\to 0}\PP[K_1\cap A=\emptyset, K_2\cap A=\emptyset\cond (e^{i\eps}K_1\cap\A_r)\cap (e^{i2\eps}K_2\cap\A_r)=\emptyset]}\\
&=&\lim_{\eps\to 0,r\to 0}\frac{\PP[K_1\cap A=\emptyset, K_2\cap A=\emptyset, (e^{i\eps}K_1\cap\A_r)\cap (e^{i2\eps}K_2\cap\A_r)=\emptyset]}{\PP[(e^{i\eps}K_1\cap\A_r)\cap (e^{i2\eps}K_2\cap\A_r)=\emptyset]}\\
&=&\lim_{\eps\to 0,r\to 0}|\Phi_A'(0)|^{\xi(\beta_1)+\xi(\beta_2)}\Phi_A'(1)^{\beta_1+\beta_2}\\
&&\quad \times\frac{\PP[(e^{i\eps}K_1\cap\A_r)\cap (e^{i2\eps}K_2\cap\A_r)=\emptyset\cond K_1\cap A=\emptyset, K_2\cap A=\emptyset]}{\PP[(e^{i\eps}K_1\cap\A_r)\cap (e^{i2\eps}K_2\cap\A_r)=\emptyset]}\\
&=&\lim_{\eps\to 0,r\to 0}|\Phi_A'(0)|^{\xi(\beta_1)+\xi(\beta_2)}\Phi_A'(1)^{\beta_1+\beta_2}\\
&&\quad \times\frac{\PP[(\Phi_A(e^{i\eps})K_1\cap\Phi_A(\A_r\setminus A)\cap (\Phi_A(e^{i2\eps})K_2\cap\Phi_A(\A_r\setminus A))=\emptyset]}{\PP[(e^{i\eps}K_1\cap\A_r)\cap (e^{i2\eps}K_2\cap\A_r)=\emptyset]}\\
&=&|\Phi_A'(0)|^{\xi(\beta_1)+\xi(\beta_2)}\Phi_A'(1)^{\beta_1+\beta_2}\Phi_A'(1)^{\hat{\xi}(\beta_1,\beta_2)}|\Phi_A'(0)|^{\xi(\beta_1,\beta_2)-\xi(\beta_1)-\xi(\beta_2)}\\
&=&|\Phi_A'(0)|^{\xi(\beta_1,\beta_2)}\Phi_A'(1)^{\tilde{\xi}(\beta_1,\beta_2)}.
\end{eqnarray*}
\end{proof}